\documentclass[12pt]{amsart}

\usepackage{amssymb}
\usepackage{exscale}

\usepackage{times}

\calclayout

\usepackage{pdfsync}

\usepackage[colorlinks=true, bookmarks=true,pdfstartview=FitV, linkcolor=blue, citecolor=blue, urlcolor=blue]{hyperref}

\newtheorem{theorem}{Theorem}[section]
\newtheorem{lemma}[theorem]{Lemma}
\newtheorem{prop}[theorem]{Proposition}

\theoremstyle{definition}
\newtheorem{definition}[theorem]{Definition}

\theoremstyle{remark}
\newtheorem{remark}[theorem]{Remark}

\numberwithin{equation}{section}

\allowdisplaybreaks

\newcommand{\R}{\mathbb R}

\newcommand{\Z}{\mathbb Z}

\newcommand{\subRn}{{{\mathbb R}^n}}

\DeclareMathOperator{\supp}{supp}

\DeclareMathOperator*{\essinf}{ess\,inf}
\DeclareMathOperator*{\esssup}{ess\,sup}

\newcommand{\pp}{{p(\cdot)}}
\newcommand{\cpp}{{p'(\cdot)}}
\newcommand{\Lp}{L^{p(\cdot)}}
\newcommand{\Pp}{\mathcal P}
\newcommand{\qq}{{q(\cdot)}}

\newcommand{\Hp}{H^{p(\cdot)}}

\newcommand{\Ss}{\mathcal{S}}

\newcommand{\Mg}{\mathcal{M}_N}
\newcommand{\TMg}{\mathcal{M}_N^{\epsilon, L}}

\newcommand{\Mt}{M_{\Phi,T}}
\newcommand{\TMt}{M_{\Phi, T}^{\epsilon, L}}

\newcommand{\Mone}{M_{\Phi, 1}}
\newcommand{\TMone}{M_{\Phi, 1}^{\epsilon, L}}

\newcommand{\Mr}{M_{\Phi, 0}}
\newcommand{\TMr}{M_{\Phi, 0}^{\epsilon, L}}

\newcommand{\Mp}{\mathcal{N}}

\newcommand{\ds}{\displaystyle}
\newcommand{\p}{\partial}
\newcommand{\MP}{M\mathcal{P}_0}

\def\Xint#1{\mathchoice
   {\XXint\displaystyle\textstyle{#1}}%
   {\XXint\textstyle\scriptstyle{#1}}%
   {\XXint\scriptstyle\scriptscriptstyle{#1}}%
   {\XXint\scriptscriptstyle\scriptscriptstyle{#1}}%
   \!\int}
\def\XXint#1#2#3{{\setbox0=\hbox{$#1{#2#3}{\int}$}
     \vcenter{\hbox{$#2#3$}}\kern-.5\wd0}}

\def\avgint{\Xint-}

\begin{document}

\title{Variable Hardy Spaces}

\author{David Cruz-Uribe, SFO}
\address{Department of Mathematics, Trinity College}
\email{David.CruzUribe@trincoll.edu}

\author{Li-An Daniel Wang}
\address{Department of Mathematics, Trinity College}
\email{Daniel.Wang@trincoll.edu}

\thanks{The first author is supported by the Stewart-Dorwart faculty
  development fund at Trinity College and by
  grant MTM2012-30748 from the Spanish Ministry of Science and
  Innovation}

\subjclass[2010]{42B25, 42B30, 42B35}

\keywords{Hardy spaces, variable Lebesgue spaces, grand maximal
  operator, atomic decomposition, singular integral operators}

\date{November 15, 2012}

\begin{abstract}
We develop the theory of variable exponent Hardy spaces $H^\pp$.
Analogous to the classical theory,   we give equivalent definitions in
terms of maximal operators.  We also show that $H^\pp$ functions have an atomic decomposition including a
``finite'' decomposition; this decomposition is more like the
decomposition for weighted Hardy spaces due to Str\"omberg and
Torchinsky~\cite{MR1011673} than the classical atomic decomposition.  As an application
of the atomic decomposition we show that singular integral operators are
bounded on $H^\pp$ with minimal regularity assumptions on the exponent
$\pp$. 
\end{abstract}

\maketitle

\section{Introduction}

Variable Lebesgue spaces are a generalization of the classical $L^p$
spaces, replacing the constant exponent $p$ with an exponent function
$\pp$:  intuitively, they consist of all functions $f$ such that
\[ \int_\subRn |f(x)|^{p(x)}\,dx < \infty.  \]
These spaces were introduced by Orlicz~\cite{0003.25203} in 1931, but
they have been the subject of more intensive study since the early 1990s,
because of their  intrinsic interest, for their use in the study of
PDEs and variational integrals with nonstandard growth conditions, and
for their applications to the study of non-Newtonian fluids and to
image restoration.
(See~\cite{cruz-fiorenza-book,diening-harjulehto-hasto-ruzicka2010}
and the references they contain.)

In this paper we extend the theory of variable Lebesgue spaces by
studying the variable exponent Hardy spaces, or more simply, the
variable Hardy spaces $H^\pp$.  The classical theory of $H^p$ spaces,
$0<p\leq 1$, is well-known
(see~\cite{MR0447953,garcia-cuerva-rubiodefrancia85, grafakos08b,MR1342077,stein93})
and our goal is to replicate that theory as much as possible in this
more general setting.  This has been done in the context of analytic
functions on the unit disk by Kokilashvili and
Paatashvili~\cite{MR2294576,MR2464049}.  We are interested in the
theory of real Hardy spaces in all dimensions.  Here we give a broad
overview of our techniques and results; we will defer the precise
statement of definitions and theorems until the body of the paper.

Given an exponent function $\pp : \R^n \rightarrow (0,\infty)$, we
define the space $\Lp$;  this is a quasi-Banach space.   In the study
of variable Lebesgue spaces it is common to assume that the exponent
$\pp$ satisfies log-H\"older continuity conditions locally and at
infinity.  While these conditions will be sufficient for us, we prefer
to work with a much weaker hypothesis:  that there exists $p_0>0$ such
that the Hardy-Littlewood maximal operator is bounded on
$L^{\pp/p_0}$.   This approach was first introduced
in~\cite{MR2210118} and developed systematically
in~\cite{cruz-fiorenza-book}.  While in certain cases  weaker
hypotheses are possible, this appears to be the ``right'' universal
condition for doing harmonic analysis in the variable exponent setting.

The variable Hardy space  $H^\pp$ consists of all tempered
distributions $f$ such that $\Mg f\in \Lp$, where $\Mg$ is the grand
maximal operator of Fefferman and Stein.   We show that an equivalent
definition is gotten by replacing the grand maximal operator with a
maximal operator defined in terms of convolution with a single
Schwartz function or with the non-tangential maximal operator defined
using the Poisson kernel.    This proof follows the broad outline of
the argument in the classical case, but differs in many technical
details.  Here we make repeated use of the fact that the maximal
operator is bounded on $L^{\pp/p_0}$.

We next prove an atomic decomposition for distributions in $H^\pp$.
Given $\pp$ and $q$, $1<q\leq \infty$, we say that a function  $a(\cdot)$ is a
$(\pp,q)$ atom if there is a ball $B$ such that $\supp(a)\subset B$,
\[ \|a\|_q \leq |B|^{1/q}\|\chi_B\|^{-1}_\pp, \qquad 
\text{and }
\qquad  \int a(x)x^\alpha\,dx =0 \]
for all multi-indices $\alpha$ such that $|\alpha|$ is not too large.
We then show that $f\in H^\pp$ if and only if for $q$ sufficiently
large  there exist $(\pp,q)$ atoms $a_j$
such that 
\begin{equation} \label{eqn:intro1}
f = \sum_j \lambda_j a_j,
\end{equation}
and
\begin{equation} \label{eqn:intro2}
 \|f\|_{H^\pp} \approx \inf\left\{ 
\bigg\| \sum_j \lambda_j \frac{\chi_{B_j}}{\|\chi_{B_j}\|_\pp}
  \bigg\|_\pp : f = \sum_j \lambda_j a_j \right\},
\end{equation}
where the infimum is taken over all possible atomic decompositions of
$f$. 
This is very different from the classical atomic decomposition;
it is based on the atomic decomposition developed for
weighted Hardy spaces by Str\"omberg and Torchinsky~\cite{MR1011673}.
A comparable decomposition in the classical case is due to Uchiyama:
see Janson and Jones~\cite{MR671315}.  Moreover, we are able to prove
that for $q<\infty$, if the summation in~\eqref{eqn:intro1} is finite, the infimum in
\eqref{eqn:intro2} can be taken over finite decompositions.  This
``finite'' atomic decomposition is a generalization of the result of
Meda, {\em et al.}~\cite{MR2399059} in the classical case.   As part
of our work we also prove a finite atomic decomposition theorem for weighted Hardy
spaces, extending the results in~\cite{MR1011673}.

To construct our atomic decomposition we first adapt the 
Calder\'on-Zygmund decomposition of classical Hardy spaces to give a
$(\pp,\infty)$ atomic decomposition.  Here a key tool is a
vector-valued inequality for the maximal operator, which in turn
depends on the boundedness of maximal operator.   For the case $q<\infty$ we
also rely on the theory of weighted Hardy spaces and on the Rubio
de Francia extrapolation theory
for variable Lebesgue spaces developed in~\cite{MR2210118} (see
also~\cite{cruz-fiorenza-book,cruz-martell-perezBook}).  Neither
approach was sufficient in itself in this case. We were not able to extend the
classical approach to prove half the equivalence in~\eqref{eqn:intro2}.  On
the other hand, while such an equivalence exists in the weighted case,
extrapolation requires careful density arguments and we could not,
{\em a priori}, find the requisite dense subsets needed to prove both
inequalities in~\eqref{eqn:intro2}.   Again, in applying
extrapolation the key hypothesis is the boundedness assumption on the
maximal operator.  

Finally, we prove that convolution type Calder\'on-Zygmund singular integral operators
with sufficiently smooth kernels are bounded on $H^\pp$.  In our proof
we make extensive use of the finite atomic decomposition in weighted
Hardy spaces; this allows
us to avoid the more delicate convergence arguments that are often
necessary when using the ``infinite'' atomic decomposition (e.g.,
see~\cite{garcia-cuerva-rubiodefrancia85}). 

The remainder of this paper is organized as follows.  In
Section~\ref{section:prelim} we give precise definitions of variable
Lebesgue spaces and state a number of results we will need in the
subsequent sections.   In Section~\ref{section:maximal} we
characterize $H^\pp$ in terms of maximal operators.  In
Sections~\ref{section:density} and~\ref{section:CZ} we prove two
technical results:  that $L^1_{loc}$ is dense in $H^\pp$ and the
Calder\'on-Zygmund decomposition for distributions in $H^\pp$.  In
Section~\ref{section:p-infty-atoms} we construct the $(\pp,\infty)$
atomic decomposition, 
and in Section~\ref{section:pq-atoms} we
construct the atomic decomposition for $q<\infty$ and prove the finite
atomic decompositions for both the variable and weighed Hardy spaces. 
This second decomposition is used in Section~\ref{section:sio}, where 
we prove that singular integrals are bounded on $H^\pp$.  

\begin{remark}
As we were completing this project we learned that the variable Hardy
spaces had been developed independently by Nakai and
Sawano~\cite{Nakai20123665}.   They prove the equivalent definitions
in terms of maximal operators using another approach.  They also
define an atomic decomposition but one which is weaker than ours.
They show that $\|f\|_{H^\pp}$ is equivalent to the infimum of 
\[ \bigg\| \bigg(\sum_j \bigg(\lambda_j^{p_*}
\frac{\chi_{B_j}}{\|\chi_{B_j}\|_\pp^{p_*}}\bigg)\bigg)^{1/p_*} 
  \bigg\|_\pp, \]
where $p_*=\min(1, \essinf p(x) )$.  In particular, if $\pp$ takes on
values less than 1, this quantity is larger than that in
\eqref{eqn:intro2}.  They prove that this is equivalent
to~\eqref{eqn:intro2} only when $q=\infty$ and with the further assumption that $\pp$ is
log-H\"older continuous.   Using their atomic decomposition they prove
that singular integrals are bounded, but again they must assume that $\pp$ is
log-H\"older continuous.
\end{remark}

\section{Preliminaries}
\label{section:prelim}

In this section we give without proof some basic results about the variable Lebesgue
spaces.  Unless otherwise specified, we refer the
reader to~\cite{cruz-fiorenza-book,dcu-af-crm,diening-harjulehto-hasto-ruzicka2010,
MR1866056,MR1134951} for proofs and further information.
Let $\Pp=\Pp(\R^n)$ denote the collection of all measurable functions $\pp :
\R^n \rightarrow [1,\infty]$.   Given a measurable set $E$, let
\[ p_-(E) = \essinf_{x\in E} p(x), \qquad p_+(E) = \esssup_{x\in E}
p(x). \]
For brevity we will write $p_-=p_-(\R^n)$ and $p_+=p_+(\R^n)$.   
Define the set $\Omega_\infty = \{ x\in \R^n : p(x)=\infty \}$.
Then for $\pp\in \Pp$, the space $\Lp=\Lp(\R^n)$ is the collection of
all measurable functions $f$ such that for some $\lambda>0$,
\[ \rho(f/\lambda) = \int_{\R^n \setminus \Omega_\infty}
\left(\frac{|f(x)|}{\lambda}\right)^{p(x)}\,dx + \lambda^{-1}\|f\|_{L^\infty(\Omega_\infty)}< \infty.  \]
This becomes a Banach function space when equipped with the Luxemburg
norm
\[ \|f\|_\pp = \inf\left\{ \lambda > 0 : \rho(f/\lambda)\leq 1
\right\}. \]

Given $\pp\in \Pp$, define the conjugate exponent $\cpp$ by the
equation
\[ \frac{1}{p(x)}+\frac{1}{p'(x)} = 1, \]
with the convention that $1/\infty = 0$.

\begin{lemma} \label{lemma:holder}
Given $\pp \in \Pp$, if $f\in \Lp$ and $g\in L^\cpp$,
\[ \int_\subRn |f(x)g(x)|\,dx \leq C(\pp) \|f\|_\pp\|g\|_\cpp. \]
Conversely for all $f\in \Lp$,
\[ \|f\|_\pp \leq C(\pp)\sup \int_\subRn f(x)g(x)\,dx, \]
where the supremum is taken over all $g\in L^\cpp$ such that
$\|g\|_\cpp \leq 1$. 
\end{lemma}

\begin{lemma} \label{lemma:imbed}
  Let $E\subset \R^n$ be such that  $|E|<\infty$.  If
  $\pp,\,\qq\in\Pp$ satisfy $p(x)\leq q(x)$ a.e., then
\[ \|f\chi_E\|_\pp \leq (1+|E|)\|f\chi_E\|_\qq.
\]
\end{lemma}

To define the variable Hardy spaces we need to extend the collection
of allowable exponents.  For simplicity we restrict ourselves to spaces
where $\pp$ is bounded.
Let $\Pp_0$ denote the collection of all
measurable functions $\pp : \R^n \rightarrow (0,\infty)$ such that
$p_+<\infty$.   With the same definition of the modular $\rho$ as above, we
again define $\Lp$ as the collection of measurable functions $f$ such
that for some $\lambda>0$, $\rho(f/\lambda)<\infty$.   We define
$\|\cdot\|_\pp$ as before; if $p_-<1$ (the case we are primarily
interested in) this is not a norm:  it is a quasi-norm and
$\Lp$ becomes a quasi-Banach space.  We will abuse terminology and
refer to it as a norm.

The next four lemmas are proved exactly as in the case when $p_-\geq 1$.

\begin{lemma} \label{lemma:homog-exp}
Given $\pp\in \Pp_0$, $p_+<\infty$,
then for all $s>0$,
\[ \||f|^s\|_\pp = \|f\|_{s\pp}^s. \]
\end{lemma}

\begin{lemma} \label{lemma:norm-mod}
Suppose $\pp \in \Pp_0$.  Given a sequence
$\{f_k\}\subset \Lp$,
\[ \int_\subRn |f_k(x)|^{p(x)}\,dx \rightarrow 0 \]
as $k\rightarrow \infty$ if and only if $\|f_k\|_\pp\rightarrow 0$.
\end{lemma}

\begin{lemma} \label{lemma:monotone}
Suppose $\pp \in \Pp_0$.  Given a sequence $\{f_k\}$ of $\Lp$ functions
that increase pointwise almost everywhere to a function $f$, 
\[ \lim_{k\rightarrow \infty} \|f_k\|_\pp = \|f\|_\pp. \]
\end{lemma}

\begin{lemma} \label{lemma:norm-mod2}
Suppose $\pp \in \Pp_0$.  Given $f\in \Lp$,
if $\|f\|_\pp \leq 1$,
\[ \rho(f)^{1/p_-} \leq \|f\|_\pp \leq \rho(f)^{1/p_+}; \]
if $\|f\|_\pp \geq 1$,
\[ \rho(f)^{1/p_+} \leq \|f\|_\pp \leq \rho(f)^{1/p_-}. \]
\end{lemma}

\begin{lemma} \label{lemma:minkowski-low}
Given $\pp\in \Pp_0$, $p_-\leq 1$, then for all $f,\,g\in \Lp$,
\[ \|f+g\|_\pp^{p_-} \leq \|f\|_\pp^{p_-}+\|g\|_\pp^{p_-} \]
\end{lemma}

\begin{proof}
Since $\pp/p_-\in \Pp$, by Lemma~\ref{lemma:homog-exp}, convexity and Minkowski's
inequality for the variable Lebesgue spaces,
\begin{multline*}
 \|f+g\|_\pp^{p_-} = \||f+g|^{p_-}\|_{\pp/p_-} \leq
 \||f|^{p_-}+|g|^{p_-}\|_{\pp/p_-}  \\ \leq
\||f|^{p_-}\|_{\pp/p_-}+\||g|^{p_-}\|_{\pp/p_-} = \|f\|_\pp^{p_-}
+\|g\|_\pp^{p_-}.
\end{multline*}
\end{proof}

\begin{remark}
This lemma is false if $p_->1$, but in this case $\|\cdot\|_\pp$ is a
norm and so Minkowski's inequality holds.   This will cause minor
technical problems in the proofs below; we will generally consider the
case $p_-\leq 1$ in detail and sketch the changes required for the
other case.
\end{remark}

\subsection{The Hardy-Littlewood maximal operator}
Given a function $f\in L^1_{loc}$ we define the maximal function of
$f$ by
\[ Mf(x) =\sup_{Q\ni x} \avgint_Q |f(y)|\,dy, \]
where $\avgint_Q g\, dy = |Q|^{-1}\int_Q g\, dy$, and the supremum is
taken over all cubes whose sides are parallel to the coordinate axes.
Throughout, we will make use of the
following class of exponents.

\begin{definition} \label{defn:MP0} Given $\pp \in \mathcal{P}_0$, we
  say $\pp \in \MP$ if $p_- > 0$ and there exists $p_0$, $0 < p_0 <
  p_-$, such that $\|Mf\|_{\pp/p_0} \leq C(n,\pp,p_0)\|f\|_{\pp/p_0}$.
\end{definition}

A useful sufficient condition for the boundedness of the maximal
operator is log-H\"older continuity:  for a proof,
see~\cite{cruz-fiorenza-book,diening-harjulehto-hasto-ruzicka2010}. 

\begin{lemma}
Given $\pp \in \Pp$, such that $1<p_-\leq p_+<\infty$, suppose that
$\pp$ satisfies the log-H\"older continuity condition locally,
\begin{equation}\label{eqn:log-Holder}
|p(x) - p(y)| \leq \frac{C_0}{-\log(|x-y|)}, \qquad |x-y| < 1/2,
\end{equation}
and at infinity:  there exists $p_\infty$ such that
\begin{equation}\label{eqn:decay}
|p(x) - p_\infty| \leq \frac{C_\infty}{\log(e+|x|)}.
\end{equation}
Then $ \|Mf\|_\pp \leq C(n,\pp)\|f\|_\pp$.
\end{lemma}

\begin{remark}
 We want to stress that while in practice it is common to assume that
the exponent $\pp$ satisfies the log-H\"older continuity conditions,
we will not assume this in our main results.  For a further discussion
of sufficient conditions for the maximal operator to be bounded,
see~\cite{cruz-fiorenza-book, diening-harjulehto-hasto-ruzicka2010} and the references
they contain.
\end{remark}

\begin{lemma} \label{lemma:max-up}
Given $\pp \in \Pp$, if the maximal operator is bounded on $\Lp$, then
for every $s>1$, it is bounded on $L^{s\pp}$.
\end{lemma}

\begin{proof}
This follows at once from H\"older's inequality and
Lemma~\ref{lemma:homog-exp}:
\[  \|Mf\|_{s\pp} = \|(Mf)^s\|_\pp^{1/s} \leq \|M(|f|^s)\|_\pp^{1/s}
\leq
C^{1/s} \||f|^s\|_\pp^{1/s}=C^{1/s}\|f\|_{s\pp}. \]
\end{proof}

The following  necessary condition is due to
Kopaliani~\cite{MR2341730}.  It should be compared to the Muckenhoupt
$A_p$ condition from the study of weighted norm inequalities.
(See~\cite{duoandikoetxea01,garcia-cuerva-rubiodefrancia85}.)

\begin{lemma} \label{lemma:kopaliani}
Given $\pp \in \Pp$, if the maximal operator is bounded on $\Lp$,
then for every ball $B\subset \R^n$,
\[ \|\chi_B\|_\pp \|\chi_B\|_\cpp \leq C|B|. \]
\end{lemma}

The maximal operator also satisfies a vector-valued inequality.   This
result was proved using extrapolation
in~\cite{MR2210118}.  (See also~\cite{cruz-fiorenza-book,cruz-martell-perezBook}.)

\begin{lemma} \label{lemma:max-vector}
Given $\pp \in \Pp$ such that $p_+<\infty$, if the maximal operator is bounded on $\Lp$, then
 for any
$r$, $1<r<\infty$,
\[ \bigg\| \bigg(\sum_k (Mf_k)^r\bigg)^{1/r}\bigg\|_\pp
\leq C(n,\pp,r) \bigg\| \bigg(\sum_k |f_k|^r\bigg)^{1/r}\bigg\|_\pp. \]
\end{lemma}

Our final lemma is a deep result due to
Diening~\cite{MR2166733,diening-harjulehto-hasto-ruzicka2010}.  

\begin{lemma} \label{lemma:diening}
Given $\pp\in \Pp$ such that $1< p_- \leq p_+<\infty$, the maximal operator is bounded
on $L^\pp$ if and only if it is bounded on $L^\cpp$.  
\end{lemma}

\section{The maximal characterization}
\label{section:maximal}

In this section we define the variable Hardy spaces and give
equivalent characterizations in terms of maximal operators.  To state
our results, we need a few definitions. Let $\Ss$ be the space of Schwartz functions and let
$\Ss'$ denote the space of tempered distributions.   We will say that
a tempered distribution $f$ is bounded if $f*\Phi\in L^\infty$ for
every $\Phi\in \Ss$.    For complete
information on distributions, see~\cite{MR1681462, MR0304972}.
Define the family of semi-norms on $\|\cdot\|_{\alpha,\beta}$,
$\alpha$ and $\beta$ multi-indices, on
$\Ss$ by
\[ \|f\|_{a,b} = \sup_{x\in \R^n}  |x^\alpha D^\beta f(x)|, \]
and for each integer $N>0$ let
\[ \Ss_N = \big\{ f \in \Ss : \|f\|_{\alpha,\beta} \leq 1,
|\alpha|,|\beta| \leq N \big\}. \]

Given $\Phi$ and $t>0$, let $\Phi_t(x)=t^{-n}\Phi(x/t)$.  
We define three maximal operators:  given $\Phi\in \Ss$ and $f\in \Ss'$,
define the radial maximal operator
 \[    \Mr f = \sup_{t > 0} |f \ast \Phi_t (x)|, \]
and for each $N>0$ the grand maximal operator,
\[     \Mg f (x) = \sup_{\Phi \in \Ss_N} \Mr f(x). \]
Finally, define the non-tangential maximal operator
\[ \Mp f(x) = \sup_{|x-y|<t} |P_t*f(y)|, \]
where $P$ is the Poisson kernel
\[ P(x) =
\frac{\Gamma\left(\frac{n+1}{2}\right)}{\pi^{\frac{n+1}{2}}}\frac{1}{(1+|x|^2)^{\frac{n+1}{2}}}. \]

Our main result in this section is the following.

\begin{theorem} \label{Thm-Max}
Given $\pp \in \MP$, for every $f\in \Ss'$ the following are equivalent:
\begin{enumerate}

\item there exists $\Phi\in \Ss$, $\int \Phi(x) dx \neq 0$, such that
  $\Mr f \in \Lp$;

\item for all $N>n/p_0 +n+1$, $\Mg f\in \Lp$;

\item $f$ is a bounded distribution and $\Mp f \in \Lp$.

\end{enumerate}
Moreover, the quantities $\|\Mr f\|_\pp$, $\|\Mg f\|_\pp$ and $\|\Mp
f\|_\pp$ are comparable with constants that depend only on $\pp$ and
$n$ and not on $f$.
\end{theorem}

If we choose $N$ sufficiently large, then by Theorem~\ref{Thm-Max} we
can use any of these three maximal operators to given an equivalent
definition of the variable Hardy spaces.  To be definite we will use
the grand maximal operator, but in the rest of the paper we will 
move between these three norms without comment.

\begin{definition} \label{Hp-defn}
Let  $\pp \in \MP$. For $N>n/p_0 +n+1$, define the space $H^\pp$ to be the collection of
$f\in \Ss'$ such that $\|f\|_{H^\pp}=\|\Mg f\|_\pp<\infty$. 
\end{definition}

\begin{proof}[Proof of Theorem \ref{Thm-Max}]
The proof is similar to that of the corresponding
result for real Hardy spaces:  cf.~\cite{MR1342077, stein93}.  
The most difficult step is the implication
$(1)\Rightarrow (2)$ which we will prove in
Sections~\ref{subsec:12a} and~\ref{sec:apriori}.   We will then prove $(2)\Rightarrow (1)$ in Section \ref{subsec:21} and $(2)
\Rightarrow (3) \Rightarrow (1)$ in Section~\ref{subsec:2-3}.

\subsection{The implication $(1)\Rightarrow (2)$}
\label{subsec:12a}

The proof requires two supplemental operators:  the non-tangential maximal operator with aperture $1$,
\[     \Mone f(x) = \sup_{\substack{|x - y| < t\\ t > 0}} |f \ast \Phi_t (y)|, \]
and the tangential maximal operator,
\[     \Mt f(x) = \sup_{\substack{y \in \R^n\\ t > 0}} |\Phi_t \ast f(x - y)|
\left( 1 + \frac{|y|}{t} \right)^{-T}. \]
Note that $T$ is a parameter in the definition of $\Mt$ and not just
notation indicating the that this is a ``tangential'' operator.

We will prove this implication by proving three norm
inequalities.
First, if
$N\geq T+n+1$, we will show that
\begin{equation} \label{eqn:Step-a}
\|\Mg f\|_\pp \leq C(n,\Phi) \|\Mt f\|_\pp.
\end{equation}
Second, if $T>n/p_0$, we will show that
\begin{equation} \label{eqn:Step-cd}
\|\Mt f \|_\pp \leq C(n,T,\pp,p_0)\|\Mone f \|_\pp.
\end{equation}
Finally, we will show that
\begin{equation} \label{eqn:Step-e}
 \|\Mone f\|_\pp \leq C(\pp,T)\|\Mr f\|_\pp.
\end{equation}
To prove this we will first make the {\em a priori} assumption that
$\Mone f\in \Lp$; we will then show that this is always the case by
showing that if $\Mr f\in
\Lp$, then \eqref{eqn:Step-e} holds with a constant that depends on
$f$.   This proof parallels the proof we just sketched; to emphasize
this we will defer it to Section~\ref{sec:apriori} and organize it
similarly.

\subsubsection*{Proof of inequality \eqref{eqn:Step-a}}

The proof requires a lemma from~\cite[Lemma~2.1]{MR1342077}.

\begin{lemma} \label{lemma:lu}
Let $\Phi \in \Ss$, $\int \Phi(x)\,dx \neq 0$.  Then for any $\Psi \in
\Ss$ and $T>0$, there exist functions $\Theta^s \in \Ss$, $0<s<1$,
such that
\[ \Psi(x) = \int_0^1 \Phi_t * \Theta^s(x)\,dx \]
and for all $m\geq T+1$,
\[ \int_\subRn (1+|x|)^T |\Theta^s(x)|\,dx \leq C(\Phi,n) s^T
\|\Psi\|_{m+n,m}. \]
\end{lemma}

Fix $N\geq T+n+1$ and fix $\Psi \in S_N$.  Then by the definition of
the tangential maximal operator, by making the change of variables
$w=z/t$, and by Lemma~\ref{lemma:lu}, we get
\begin{align*}
|f*\Psi_t(x)|
& \leq \int_0^1 \int_\subRn
|f*\Phi_{st}(x-z)||\Theta^s(z/t)|t^{-n}\,dz\,ds \\
& = \int_0^1 \int_\subRn
|f*\Phi_{st}(x-z)|\left(1+\frac{|z|}{st}\right)^{-T} \\
& \qquad \qquad \times \left(1+\frac{|z|}{st}\right)^{T}|\Theta^s(z/t)|t^{-n}\,dz\,ds \\
& \leq \Mt f(x) \int_0^1 \int_\subRn \left(\frac{1}{s}+\frac{|z|}{st}\right)^{T}|\Theta^s(z/t)|t^{-n}\,dz\,ds \\
& \leq \Mt f(x) \int_0^1 s^{-T} \int_\subRn
(1+|w|)^T|\Theta^s(w)|\,dw\,ds \\
& \leq C(\Phi,n) \Mt f(x)\|\Psi\|_{T+n+1,T+1} \\
& \leq C(\Phi,n) \Mt f(x).
\end{align*}
Given this pointwise inequality, we immediately get
inequality~\eqref{eqn:Step-a}.

\subsubsection*{Proof of inequality~\eqref{eqn:Step-cd}}
Our proof is adapted
from~\cite[Lemma~3.1]{MR1342077}.    Fix $x,\,y\in \R^n$ and $t>0$.
Then for all $z\in B(x-y,t)$,
\[ |f*\Phi_t(x-y)| \leq \Mone f(z).  \]
Let $q=n/T>0$.  Since $B(x-y,t) \subset B(x,|y|+t)$, we have that
\begin{multline*}
|f*\Phi_t(x-y)|^q
 \leq \avgint _{B(x-y,t)} \Mone f(z)^q\,dz \\
 \leq \frac{B(x,|y|+t)}{B(x-y,t)|} \avgint _{B(x,|y|+t)} \Mone
f(z)^q\,dz 
 \leq \left(1+\frac{|y|}{t}\right)^n M(\Mone(f)^q)(x).
\end{multline*}
If we rearrange terms, then by our choice of $q$ we have that
\[ \left|f*\Phi_t(x-y)\left(1+\frac{|y|}{t}\right)^{-T}\right|^q \leq
M(\Mone(f)^q)(x). \]
Taking the supremum over all $y$ and $t$ we get that
\[ \Mt f(x)^q \leq M(\Mone(f)^q)(x). \]
Therefore, by Lemmas~\ref{lemma:homog-exp} and~\ref{lemma:max-up}, 
since $\pp\in \MP$ and $q=n/T<p_0$,
\begin{multline*}
 \|\Mt f\|_\pp = \|(\Mt f)^q\|_{\pp/q}^{1/q} \leq
\|M(\Mone(f)^q)\|_{\pp/q}^{1/q} \\
\leq C(\pp,p_0,n,q) \|\Mone(f)^q\|_{\pp/q}^{1/q} = C(\pp,p_0,n,q)\|\Mone f\|_\pp.
\end{multline*}
Since $q$ depends on $n$ and $T$, this gives us inequality~\eqref{eqn:Step-cd}.

\subsubsection*{Proof of inequality~\eqref{eqn:Step-e}}
As we remarked above, we first assume that $\Mone f \in \Lp$.  Our
argument is very similar to that in Stein~\cite[pp.~95--98]{stein93}.

Let $\lambda > 0$ be some large number; the precise value will be
fixed below.  Define $F = F_{\lambda} = \{ x : \Mg f (x) \leq \lambda
\Mone f (x) \}$.
Then by inequalities~\eqref{eqn:Step-a} and~\eqref{eqn:Step-cd},
\[
   \| \Mone (f) \cdot \chi_{F^c} \|_{\pp}
        \leq \frac{1}{\lambda} \| \Mg (f) \chi_{F^c} \|_{\pp} 
\leq        \frac{1}{\lambda} \| \Mg (f) \|_{\pp}
\leq \frac{C_0}{\lambda} \| \Mone (f) \|_{\pp},
\]
where $C_0=C_0(n,\Phi,T,\pp,p_0)$.  Therefore, by
Lemma~\ref{lemma:minkowski-low} (if $p_-<1$; the other case is treated similarly),
\begin{multline*}
    \| \Mone f \|_{\pp}^{p_-}
        \leq \| \Mone (f) \cdot \chi_{F} \|_{\pp}^{p_-} + \| \Mone (f) \cdot \chi_{F^c} \|_{\pp}^{p_-} \\
        \leq \| \Mone (f) \cdot \chi_F \|_{\pp}^{p_-} + \left( \frac{C_0}{\lambda} \right)^{p_-} \| \Mone f \|_{\pp}^{p_-}.
\end{multline*}
Fix $\lambda = 2^{1/p_-}C_0$;  since we assumed that $\Mone f \in
L^{\pp}$, we can rearrange terms to get
    \[  \| \Mone f \|_{\pp} \leq 2 \| \Mone (f) \cdot \chi_F
    \|_{\pp}.   \]

    To estimate the right-hand side, we will use the fact that there exists $c = c(p_0,\Phi,n,N,\lambda)$ such that for all
    $x\in F$,
\begin{equation}    \label{Mone-MHL}
    \Mone f (x) \leq c M ((\Mr f)^{p_0}) (x) ^{1/p_0}.
  \end{equation}
(See~\cite[p.~96]{stein93}.)   Then again by
Lemma~\ref{lemma:homog-exp} and since $\pp\in \MP$,
\begin{multline*}
    \| \Mone f \cdot \chi_F \|_{\pp}
        \leq c \| (M((\Mr f)^{p_0}) )^{1/p_0} \|_{\pp} \\
        \leq \|M((\Mr f)^{p_0}) \|_{\pp/p_0}^{1/p_0}
        \leq C \| (\Mr f)^{p_0} \|_{\pp/p_0}^{1/p_0} = C\| \Mr f \|_{\pp},
\end{multline*}
where $C=C(p_0,\Phi,n,N,\lambda,\pp,p_0)$.  This completes the proof of
inequality~\eqref{eqn:Step-e} with the {\em a priori} assumption that
$\Mone f \in \Lp$.

\medskip

\subsection{Proof that $\Mone f\in \Lp$}
\label{sec:apriori}

To show that this {\em a priori} assumption holds, we adapt the argument
briefly sketched in~\cite[p.~97]{stein93}. 
Fix a tempered distribution $f$ such that $\Mr f\in \Lp$.  We define  truncated versions of the operators
used in the previous section.  
Hereafter, $\epsilon>0$ will be a positive parameter that will
tend to $0$, and $L>0$ will be a constant that will depend on $f$ but
will be independent of $\epsilon$.  Define
\begin{align*}
\TMr 
       & = \sup_{t\geq 0} |f\ast \Phi_t(x)| \frac{t^L}{(t + \epsilon + \epsilon|x|)^L}, \\
    \TMg f(x)
        &= \sup_{\Phi \in S_N} M_{\Phi, 0}^{\epsilon, L} f(x), \\
    \TMone f(x)
        &= \sup_{\substack{0 < t < 1/\epsilon\\ |x - y| < t}} |f \ast \Phi_t (y)| \frac{t^L}{(t + \epsilon + \epsilon|y|)^L} \\
    \TMt f(x)
        &= \sup_{\substack{y \in \R^n\\ 0 < t < 1/\epsilon}}
|\Phi_t \ast f(x - y)| \left( 1 + \frac{|y|}{t} \right)^{-T} \frac{t^L}{(t + \epsilon + \epsilon|x - y|)^L}.
\end{align*}
For every $t$ and $L$, $ \frac{t^L}{(t + \epsilon + \epsilon |x|)^L}$
increases to $1$ as $\epsilon\rightarrow 0$; in particular  $\TMone f$
increases pointwise to $\Mone f$.

Our proof proceeds as follows.  We start by showing that there exists
$L=L(f,n,p_0)>0$ such that for every $\epsilon$, $0<\epsilon<1/2$,
$\TMone f \in \Lp$.  Next we will prove three inequalities.  First, we
will show that there exists $N= T+L+n+1$ such that for all
$\epsilon>0$, 
\begin{equation} \label{eqn:Step-a1}
\| \TMg f\|_\pp \leq C(\phi,n)\|\TMt f\|_\pp.
\end{equation}
We will then show that
\begin{equation} \label{eqn:Step-cd1}
 \|\TMt f\|_\pp \leq C(n,T,\pp,p_0)\|\TMone f \|_\pp.
\end{equation}
Finally, we will show that if
\[ x \in F = F_\lambda^{\epsilon,L} = \{ x : \TMg f(x) < \lambda
\TMone f(x) \}, \]
then
\begin{equation} \label{eqn:Step-e1}
 \TMone f(x) \leq C(p_0,\Phi,n,N,L,\lambda)M(\Mr(f)^{p_0})(x)^{1/p_0}.
\end{equation}

We can then repeat the argument used to prove
inequality~\eqref{eqn:Step-e}.  First,
using~\eqref{eqn:Step-a1}, \eqref{eqn:Step-cd1} and the fact that
$\TMone\in \Lp$, we show that there exists
$\lambda=\lambda(\Phi,n,T,\pp,p_0)$ such that 
\[ \|\TMone f\|_\pp \leq 2\|\Mone(f) \chi_F\|_\pp. \]
Then we can use \eqref{eqn:Step-e1} to show that 
\[ \|\TMone f\|_\pp \leq C\|\Mr f\|_\pp < \infty. \]
The constant $C$ is independent of $\epsilon$, and so by Fatou's lemma, we get that
\[ \|\Mone f\|_\pp \leq  C(f)\|\Mr f\|_\pp < \infty. \]
This completes the proof.

\subsubsection*{Construction of the constant $L=L(f)$}

 Since $f\in \Ss'$, it is a
continuous linear functional on $\Ss$.   In particular,  arguing as in Folland~\cite[Proposition~9.10]{MR1681462}, we have
that  there exists $m>0$ (depending only on $f$) such that
\[ |f*\Phi_t(y)| \leq C(\Phi,f) \left(1+\frac{|y|}{t}\right)^m.  \]
Assume $L>2m$.  If $x$, $y$ and $t$ are such that $|x-y|<t<1/\epsilon$, then we
have that
\[
 \left(1+\frac{|y|}{t}\right)^m\frac{t^L}{(t+\epsilon+\epsilon|y|)^L}
\leq \epsilon^{-L}\left(\frac{1}{\epsilon}+\frac{|y|}{t}\right)^{m-L} \leq
 \epsilon^{-L}\left(\frac{1}{\epsilon}+\frac{|y|}{t}\right)^{-L/2}. \]
But by the triangle inequality,
\[ 1+\epsilon|x| < 1+\frac{|x|}{t} < 2+\frac{|y|}{t} <
\frac{1}{\epsilon}+\frac{|y|}{t}.  \]
Combining these inequalities we get that
\[ |f*\Phi_t(y)| \frac{t^L}{(t+\epsilon+\epsilon|y|)^L}  \leq
 \epsilon^{-L} C(\Phi,f) (1+\epsilon|x|)^{-L/2}. \]
Fix $x$ and take the supremum over all such $y$ and $t$; this shows
that
\[ \TMone f(x) \leq \epsilon^{-L} C(\Phi,f) (1+\epsilon|x|)^{-L/2}. \]
Finally, recall that
\[ (1+\epsilon|x|)^{-n} \leq \epsilon^{-n}(1+|x|)^{-n} 
\leq \epsilon^{-n}C(n) M(\chi_{B(0,1)})(x); \]
hence,
\[ \TMone f(x) \leq \epsilon^{-3L/2} C(\Phi,f,n)
M(\chi_{B(0,1)})(x)^{\frac{L}{2n}}. \]
Fix  $L$ so that $\frac{L}{2n}>\frac{1}{p_0}$.  Since $\pp\in \MP$, by Lemmas~\ref{lemma:homog-exp} and~\ref{lemma:max-up},
\[ \|\TMone f\|_\pp \leq
C\|M(\chi_{B(0,1)}\|_{\frac{L}{2n}\pp}^{\frac{L}{2n}} \leq
C\|\chi_B(0,1)\|_{\frac{L}{2n}\pp}^{\frac{L}{2n}}<\infty, \]
where $C=C(\Phi,f,n,\epsilon,L,\pp,p_0)$.Even though this
constant depends on $\epsilon$ it does not effect the above argument,
which only used the qualitative fact that $\|\TMone f\|_\pp<\infty$.

\subsubsection*{Proof of inequality~\eqref{eqn:Step-a1}}

We begin with an auxiliary estimate.  Fix $s$, $0<s<1$ and as we did
above, assume that $\epsilon<1/2$.  Then we have that
\begin{multline*}
 \left( \frac{(s t + \epsilon + \epsilon |x - z|)^L}{(s t)^L} \right) \cdot
\frac{t^L}{(t + \epsilon + \epsilon |x|)^L}   
 = \left( \frac{ s t + \epsilon + \epsilon |x - z|}{s} 
\cdot \frac{1}{t + \epsilon + \epsilon |x|} \right)^L \\
\leq s^{-L} \left( \frac{t + \epsilon + \epsilon |x|}{t + \epsilon +
    \epsilon |x|}
+ \frac{\epsilon|z|}{t + \epsilon + \epsilon |x|} \right)^L 
\leq\left( \frac{1}{s} + \frac{|z|}{st} \right)^L.
\end{multline*}

Given this estimate we can argue exactly as in proof of
inequality~\eqref{eqn:Step-a1}:  if $N\geq T+L+n+1$ and
$\Psi\in \Ss_N$, then we get that
\begin{align*}
& |f*\Psi_t(x)|\frac{t^L}{(t+\epsilon+\epsilon|x|)^L} \\
& \qquad \leq \int_0^1 \int_\subRn
|f*\Psi_{st}(x-z)|\frac{t^L}{(t+\epsilon+\epsilon|x|)^L}
\frac{(st)^L}{(st +\epsilon+\epsilon|x-z|)^L} \\
& \qquad \qquad  \times
\left(1+\frac{|z|}{st}\right)^{-T} \frac{(st +\epsilon+\epsilon|x-z|)^L}{(st)^L}
\left(1+\frac{|z|}{st}\right)^{T} \Theta^s(z/t)t^{-n}\,dz\,ds \\
& \qquad \leq \TMt f(x) \int_0^1 \int_\subRn
\left(\frac{1}{s}+\frac{|z|}{st}\right)^{L+T}
\Theta^s(z/t)t^{-n}\,dz\,ds \\
& \qquad \leq C(\Phi,n) \TMt f(x)  \|\Psi\|_{T+L+n+1,T+L+1} \\
& \qquad \leq C(\Phi,n) \TMt f(x).
\end{align*}
The desired inequality follows immediately.

\subsubsection*{Proof of inequality~\eqref{eqn:Step-cd1}}

This proof is a straightforward modification of the proof of
inequality~\eqref{eqn:Step-cd}.   As before, for all $x,\,y\in \R^n$,
$t>0$ and $\epsilon>0$, we have that
\[ |\Phi_t*f(x-y)| \frac{t^L}{(t+\epsilon+\epsilon|x-y|)^L} \leq
\TMone f(z) \]
for all $x\in B(x-y,t)$.  The proof now proceeds as before.

\subsubsection*{Proof of inequality~\eqref{eqn:Step-e1}}

Fix $ x \in F = \{ x : \TMg f(x) < \lambda
\TMone f(x) \}$.  Then by the definition of the truncated maximal
operator, there exists $(y, t)$  with $t <1/\epsilon$ and $|x - y| <
t$, such that
\begin{equation} \label{eqn:saturate}
\TMone f (x) \leq 2\,|f \ast \Phi_t (y)| \left(
  \frac{t^L}{(t + \epsilon + \epsilon |y|)^L} \right).
\end{equation}
Let $r>0$ be small; its precise value will be fixed below. If $x' \in
B(y, rt)$, then by the Mean Value Theorem,
\[ |f(x', t) - f(y, t)| \leq rt \sup_{|z - y| < rt} |\nabla_z f(z,
t)|,\]
where for brevity we write $f(y,t)=\Phi_t*f(y)$.  

Inequality \eqref{eqn:Step-e1} follows if we can prove that there
exists $c = c(N, L,n,\Phi)$ such that
\begin{align}        \label{Grad}
        t \sup_{|z - y| < rt} |\nabla_z f(z, t)| 
\leq c \TMg f (x) \cdot \frac{(t + \epsilon + \epsilon |y|)^L}{t^L}.
\end{align}
For if \eqref{Grad} holds, then for $x \in F$, and $x' \in B(y,rt)$,
\[
        |f(x', t) - f(y, t)|
            \leq c r\TMg f (x) \cdot \frac{(t + \epsilon + \epsilon |y|)^L}{t^L} 
            \leq c r \lambda \TMone f (x) \cdot \frac{(t + \epsilon + \epsilon |y|)^L}{t^L}.
\]
Now fix $r = r(N,L,n,\Phi, \lambda)$ so small that $cr\lambda \leq 1/4$. Then we have
\begin{multline*}
        |f(x', t)|
            \geq \left| \ |f(y, t)| - c r\lambda \TMone f (x) 
\cdot \frac{(t + \epsilon + \epsilon |y|)^L}{t^L} \right| \\
            \geq \left( \frac{1}{2} - c r\lambda \right) \TMone f (x)
            \geq \frac{1}{4} \TMone f(x).
        \end{multline*}
We can now get \eqref{eqn:Step-e1} by taking the average over all such
points $x'$:
        \begin{align*}
        \TMone f (x)^{p_0}
            &\leq 4^{p_0} |f(x', t)|^{p_0}\\
            & = 4^{p_0} \frac{1}{|B(y, rt)|} \int_{B(y, rt)} |f(x', t)|^{p_0} dx' \\
            &\leq 4^{p_0} \left(\frac{r + 1}{r} \right)^n \frac{1}{|B(x, (1 + r)t)|} \int_{B(x, (1 + r)t)} |f(x', t)|^{p_0} dx' \\
            &\leq c(p_0, r, n) \frac{1}{|B(x, (1 + r)t)|} \int_{B(x, (1 + r)t)} \Mr f (x')^{p_0} dx' \\
            &\leq c(p_0, r, n) M (\Mr (f)^{p_0}) (x).
        \end{align*}

 \medskip 

 To complete the proof it remains to show \eqref{Grad}. We begin with
 some notation: if we set $\Phi^{(i)}= \frac{\p \Phi}{\p z_i}$, and
 $\Phi_t^{(i)} (z) = (\Phi^{(i)})_t (z)$, then $\frac{\p}{\p z_i}
 (\Phi_t) (z) = \frac{1}{t} \Phi_t^{(i)} (z)$. Since $f \ast \Phi \in
 C^{\infty}$ whenever $f \in \mathcal{S}'$ and $\Phi \in \mathcal{S}$,
 differentiating the convolution gives
\[         \frac{\p}{\p z_i} [f(z, t)] = f \ast \frac{\p}{\p z_i}(\Phi_t)
        (z) = \frac{1}{t} f \ast \Phi^{(i)}_t (z).  \]
Hence, we can rewrite the gradient term as
        \begin{align*}
        |t \nabla_z (f) (z, t)|
            &= \left( \sum_{i = 1}^n |f \ast \Phi_t^{(i)} (z)|^2 \right)^{1/2}.
        \end{align*}
        We multiply and divide the left-hand side by the terms needed
        to obtain the truncated operator:
\[       
        t|\nabla_z f (z, t)| 
            = \underbrace{t|\nabla_z f (z, t)| \cdot \frac{t^L}{(t +
                \epsilon + \epsilon |z|)^L}}_{S(z, t)}
\cdot \underbrace{ \left( \frac{(t + \epsilon + \epsilon |z|)^L}{(t +
      \epsilon + \epsilon |y|)^L} \right)}_{R(z, y)^L}
\cdot \left( \frac{(t + \epsilon + \epsilon |y|)^L}{t^L} \right).
\]
Recall that we have fixed $x \in \R^n$ and  $(y, t)$ so that
\eqref{eqn:saturate} holds, and fixed $z \in B(y, rt)$.  Without loss of
generality, we may assume that $r\leq 1$.    We first estimate
$S(z,t)$:
 \begin{multline*}
S(z, t) = |t \nabla_z f(z, t)| \cdot \frac{t^L}{(t + \epsilon + \epsilon |z|)^L} 
 = \left( \sum_{i = 1}^n |f \ast \Phi_t^{(i)} (z)|^2 \right)^{1/2} 
\cdot \frac{t^L}{(t + \epsilon + \epsilon |z|)^L} \\
 = \left( \sum_{i = 1}^n \left| f \ast \Phi_t^{(i)} (z)
                \frac{t^L}{(t + \epsilon + \epsilon |z|)^L} \right|^2
            \right)^{1/2} \\
 \leq C(N,\Phi)\left(\sum_{i = 1}^n |M_{\Phi^{(t)}, 1}^{\epsilon, L} f (x)|^2
\right)^{1/2} 
\leq c(N, \Phi,n) \TMg f(x).
\end{multline*}
To see the first inequality, define the set of functions $\Psi=\Psi^{i,h}$
by $\Psi^{i,h}(x)= \Phi^{(i)}(x+h)$, $1\leq i \leq n$, $|h|\leq 2$.
Since $z=x+th$ for some $h$ such that $|h|\leq 1+r \leq 2$, we have
that $f*\Phi_t^{i}(z)=f*\Psi^{i,h}_t(z)$.  Moreover, since the
collection of functions $\Psi^{i,h}$ is sequentially compact in $\Ss$,
there exists a constant $c=c(\Phi,N)$ such that
$\|\Psi^{i,h}\|_{\alpha,\beta}\leq c$, $|\alpha|,\,|\beta|\leq N$.
Hence, $c^{-1}\Psi^{i,h}\in \Ss_N$ and the desired inequality follows.

\smallskip

To estimate $R(z, y)$ we note that if $z \in B(y, rt)$, then $|z| <
|y| + rt$. Then, since we may assume that  $\epsilon, r < 1$,
\[    R(z, y) \leq \left( \frac{t + \epsilon + \epsilon(|y| + rt)}{t +
                \epsilon + \epsilon |y|} \right) 
= 1 + \frac{\epsilon rt}{t + \epsilon + \epsilon |y|}
 \leq 1 + \frac{\epsilon r t}{t} = 1 + \epsilon r \leq 2.
\]
    Taking the supremum over $z$, we get
        \[ \sup_{|z - y| < rt} t |\nabla_z f(z, t)| \leq C 2^L \TMg f (x) \cdot \frac{(t + \epsilon + \epsilon |y|)^L}{t^L},
        \]
where $C=C(N,\Phi,n)$.  This gives us \eqref{Grad}.
\end{proof}

\subsection{The implication $(2)\Rightarrow(1)$}
\label{subsec:21}

Given any $\Phi \in \Ss$, there exists $c=c(\Phi)$ such that $c\Phi
\in \Ss_N$.  Therefore, the radial maximal operator is always
dominated pointwise by a constant multiple of the grand maximal
operator; hence, 
\begin{equation} \label{eqn:2imp1}
\|\Mr f\|_\pp \leq C(\Phi)\|\Mg f\|_\pp.
\end{equation}

\subsection{The implication $(2) \Rightarrow (3) \Rightarrow (1)$}
\label{subsec:2-3}

Suppose first that $(2)$ holds.  Fix $f\in \Hp$ and let $\Phi \in
\Ss$; then by \eqref{eqn:Step-e} and \eqref{eqn:2imp1}, $\Mone f
\in \Lp$.  Moreover, for all $x\in \R^n$,
\begin{multline*}
|f*\Phi(x)|^{p_0} \leq \inf_{|x-y|\leq 1} \Mone f(y)^{p_0}
\leq C(n)\int_{B(x,1)} \Mone f(y)^{p_0}\,dy \\
\leq C(n,\pp,p_0) \|(\Mone
f)^{p_0}\|_{\pp/p_0}\|\chi_{B(x,1)}\|_{(\pp/p_0)'}.
\end{multline*}
By Lemma~\ref{lemma:imbed}, if $q = \esssup_x (p(x)/p_0)'$, then
\[ \|\chi_{B(x,1)}\|_{(\pp/p_0)'} \leq (1+|B(x,1)|)
\|\chi_{B(x,1)}\|_q \leq C(n,\pp,p_0). \]
Therefore, $f*\Phi \in L^\infty$; since this is the case for every
$\Phi\in \Ss$, $f$ is a bounded distribution.

To show that $\Mp f\in \Lp$, we use the fact \cite[p.~98]{stein93} that
the Poisson kernel can be written as
\[ P(x) = \sum_{k=0}^\infty 2^{-k} \Phi_{2^k}^k(x), \]
where $\{\Phi^k\}$ is a family functions in
$\Ss$ with uniformly bounded seminorms.  Fix $x$ and $y$ such that $|x-y|<t$.  Then
\[ |f*P_t(y)| \leq \sum_{k=0}^\infty 2^{-k} |f*\Phi^k_{2^kt}(y)| \leq
\sum_{k=0}^\infty 2^{-k} M_{\Phi^k,1}f(x). \]
Taking the supremum over all such $y$ and $t$ we get that $\Mp f(x)$
is dominated by the right-hand side.  Since the functions $\Phi^k$ are
uniformly bounded, by the same argument as we used to
prove~\eqref{eqn:2imp1} we have that this inequality holds for
$\Phi^k$ with a constant independent of $k$.  Therefore, if $p_1\leq
1$, by
Lemma~\ref{lemma:minkowski-low} and~\eqref{eqn:Step-e},
\[ \|Nf\|_\pp^{p_-} \leq \sum_{k=0}^\infty 2^{kp_-}\|M_{\Phi^k,1}f\|_\pp^{p_-}
\leq C(n,\pp,p_0)\|\Mg f\|_\pp^{p_-}. \]
If $p_->1$, the same argument holds if we omit $p_-$ and
use Minkowski's inequality. 

\smallskip

Now suppose that $(3)$ holds.  Then there exists $\Phi\in
\Ss$, $\int \Phi(x)\,dx=1$ such that $\Mr f(x) \leq c\Mp f(x)$
(see~\cite[p.~99]{stein93}).   Condition $(1)$ follows immediately.

\section{Density of $L^1$ in $\Hp$}
\label{section:density}

To prove the atomic decomposition we need two facts
about variable Hardy spaces that are of interest in their own right.

\begin{prop} \label{prop:complete}
 Given  $\pp \in \MP$, $H^{\pp}$ is complete with respect to $\|\cdot\|_{H^\pp}$.
\end{prop}

\begin{proof}
  First, if a sequence $\{f_k\}$ converges in $H^{\pp}$ with respect
  $\|\cdot\|_{H^\pp}$, then it converges in $\Ss'$.  To see this, fix $\Phi\in
  \Ss$; then
\begin{multline*}
| \langle f, \Phi \rangle |^{p_0} = |f*\Phi(0)|^{p_0}
 \leq \inf_{|y|\leq 1} \Mone f(y) ^{p_0} 
 \leq C(n)\avgint_{B(0,1)} \Mone f(y) ^{p_0}\,dy  \\
 \leq C(n,\pp,p_0) \|\Mone f\|_\pp ^{p_0}
\|\chi_{B(0,1)}\|_{(\pp/p_0)'} 
\leq C(n,\pp,p_0) \|\Mone f\|_\pp ^{p_0} .
\end{multline*}
Our assertion follows at once.

To show that $H^\pp$ is complete we will consider the case $p_-\leq
1$; the case $p_->1$ is proved in essentially the same way.  It will
suffice to show that $H^\pp$  has the Riesz-Fisher property: given any
sequence $\{f_k\}$ in $H^{\pp}$ such that
\begin{equation} \label{eqn:complete1}
 \sum_k \|f_k\|_{H^\pp}^{p_-} < \infty,
\end{equation}
the series $\sum f_k$
converges in $H^\pp$.  (Cf.~Bennett and
Sharpley~\cite{bennett-sharpley88}.  The argument there is for normed
spaces but holds for quasi-norms with the
introduction of the exponent $p_-$ in \eqref{eqn:complete1}.)  Let
\[ F_j = \sum_{k=1}^j f_k; \]
then by Lemma~\ref{lemma:minkowski-low} and \eqref{eqn:complete1}, the
sequence $\{F_j\}$ is Cauchy in $H^\pp$ and so in $\Ss'$.  Therefore, it
 converges in $\Ss'$ to a tempered distribution $f$.   Moreover, we have that
\[ \|f\|_{H^\pp}^{p_-} = \big\|\sum_k f_k\big\|_{H^\pp}^{p_-} \leq
\sum_k \|f_k\|_{H^\pp}^{p_-} < \infty, \]
and so $f\in H^\pp$.  Finally,
\[ \|f-F_j \|_{H^\pp}^{p_-} \leq \sum_{k\geq j+1} \|f_k\|_{H^\pp}^{p_-}, \]
and since the right-hand side tends to $0$ as $j\rightarrow \infty$,
the series converges to $f$ in $H^\pp$.
\end{proof}
	
\medskip

\begin{prop} \label{Prop-Density}
Given  $\pp \in \MP$,  $H^{\pp} \cap L_{loc}^1$ is dense in $H^{\pp}$.
\end{prop}

\begin{proof}
  Given $f \in H^{\pp}$, by Theorem~\ref{Thm-Max},  $f$ is a bounded
  distribution.  Hence, for any $t > 0$, $f \ast P_t \in
  C^{\infty}\subset L_{loc}^1$.  Therefore, it will suffice to show $f
  \ast P_t \rightarrow f$ in $H^{\pp}$.   Again by
  Theorem~\ref{Thm-Max} it will be enough to show that as $t\rightarrow 0$,
\[ \|\Mp(f*P_t-f)\|_\pp \rightarrow 0. \]
Since $p_+<\infty$, by Lemma~\ref{lemma:norm-mod}
this will follow if
\begin{equation} \label{eqn:density1}
 \int_\subRn \Mp(f*P_t-f)(x)^{p(x)}\,dx \rightarrow 0.
\end{equation}
Since $P_s \ast P_t = P_{s + t}$, we immediately
  have
\[ \Mp(f*P_t-f)(x)^{p(x)}= \sup_{s > 0} |P_s \ast (P_t \ast f - f) (x)|^{p(x)} \leq 2^{p_+}\Mp f(x)^{p(x)} \in L^1. \]
Thus, \eqref{eqn:density1} follows from the dominated convergence
theorem if for almost every $x$,
 \begin{equation} \label{DCT-pw}
 \lim_{t \rightarrow 0} \left( \sup_{s > 0} |P_s \ast( P_t \ast f - f)(x)|\right) = 0 .
\end{equation}
To prove this, let  $u(x,s)=P_s*f(x)$.  Arguing as
above we have that  $\Mp u \in L^\pp$, and so $u(x,s)$ is
non-tangentially bounded almost everywhere.  Therefore, for almost every $x$,
\[ \lim_{s\rightarrow 0} u(x,s) \]
exists.  (See~\cite[Theorem~4.21]{garcia-cuerva-rubiodefrancia85}.)  Moreover, by
Lemmas~\ref{lemma:kopaliani} and \ref{lemma:norm-mod2} we have that for all $s$ large,
\begin{align*}
|u(x,s)|^{p_0}= |P_s * f(x)|^{p_0}
& \leq \avgint_{B(x,s)} \Mp f(y)^{p_0}\,dy \\
& \leq C(\pp) |B(x,s)|^{-1}\|Nf\|_\pp^{p_0} \|\chi_{B(x,s)}\|_{(\pp/p_0)'} \\
& \leq C(\pp,p_0) \|\chi_{B(x,s)}\|_{\pp/p_0}^{-1}\|Nf\|_\pp^{p_0} \\
& \leq C(\pp,p_0) |B(x,s)|^{-p_0/p_+}\|Nf\|_\pp^{p_0} \\
& \leq C(n,\pp,p_0) s^{-np_0/p_+}\|Nf\|_\pp^{p_0}.
\end{align*}
Thus,
\[ \lim_{s\rightarrow\infty} u(x,s)=0.  \]
These two limits, combined with the fact that $u$ is continuous, show
that for almost every $x$, $u(x,s)$ is uniformly continuous in $s$.
The limit~\eqref{DCT-pw} now follows immediately, and this completes
our proof.
\end{proof}

\section{The Calder\'on-Zygmund decomposition}
\label{section:CZ}

\begin{theorem} \label{Thm-CZ}
Given $\pp \in \MP$, fix  $f \in
L_{loc}^1 \cap H^\pp$.  For each $\lambda > 0$ define
$\Omega_{\lambda} = \{ x : \Mg f (x) > \lambda \}$.
Then there exists a set of cubes $\{Q_k^{\ast} \}$  such that
\begin{gather}
\Omega_{\lambda} = \bigcup_{k} Q_k^{\ast} \label{eqn:CZ1}\\
\intertext{and}
\sum_k \chi_{Q_k^*}(x) \leq C. \label{eqn:CZ2}
\end{gather}
Moreover, we can write $f=g + b$, where
$|g(x)| \leq c \lambda$, $b=\sum b_k$, $\supp(b_k)\subset
Q_k^*$, $\int b_k \ dx = 0$, and
\begin{equation}\label{Mr-Mg}
		\| \Mr b_k \|_{L^{\pp}(\R^n)} \leq C\| \Mg f \|_{L^{\pp}(Q_k^{\ast})}.
\end{equation}		
\end{theorem}

\begin{proof}
  Our proof is adapted from the proof for constant exponents in
  Stein~\cite{stein93}.
Since $\Omega_{\lambda}$ is open, we can form the Whitney decomposition
of $\Omega_\lambda$.   This gives us a set of cubes $\{Q_k\}$ with mutually disjoint
  interiors. Further, if we let $x_k$ and $\ell_k$ be the
  center and side length of $Q_k$, then there exist constants  $1 < \tilde{a} < a^{*}$ such that
if  $\tilde{Q}_k = \tilde{a}Q_k$ and $Q_k^{*} = a^* Q_k$,
then $Q_k \subset \tilde{Q}_k \subset Q_k^*$ and \eqref{eqn:CZ1} and
\eqref{eqn:CZ2} hold.
Let $P_0 = [-1/2, 1/2]^n$ and let $\zeta$ be a
smooth function such that $\zeta \big|_{Q_0} = 1$ and $\zeta=0$
outside $\tilde{a} Q_0$. Define $\zeta_k (x) = \zeta(\frac{x -
  x_k}{\ell_k})$ and $\eta_k = \zeta_k / (\sum_j \zeta_j)$; then
$\{ \eta_k \}$ is a partition of unity for $\Omega_{\lambda}$ subordinate to
the cover $\{ \tilde{Q}_k \}$. Lastly, define $\tilde{\eta}_k =
\eta_k / (\int \eta_k dx)$.

\medskip
Let $d=\lfloor n(1/p_0-1)\rfloor$.  We first consider the case $d\leq 0$;  then $p_-> p_0 > \frac{n}{n + 1}$, and by
Lemma~\ref{lemma:max-up} the maximal operator is bounded on
$L^{\pp\frac{n+1}{n}}$.   Let $c_k = \langle f, \tilde{\eta}_k
\rangle$ and define $b_k = (f - c_k) \eta_k$, $b = \sum_k b_k$ and $g
= f - b$.  Then $\int b_k\,dx =0$.  Moreover (see Stein~\cite[pp.~102-3]{stein93}),
$|g(x)| \leq c \lambda$ and if  $x \in Q_k^*$,
 \begin{align}
 \Mr b_k (x) &\leq C \Mg f (x);	 \label{CZ-bk}  \\
\intertext{if $x\in \R^n \setminus Q_k^*$,}
 \Mr b_k (x) &\leq C\lambda \frac{\ell_k^{n + 1}}{|x - x_k|^{n + 1}}. \label{CZ-bk2}
\end{align}

It remains to prove \eqref{Mr-Mg}.  By Lemma~\ref{lemma:minkowski-low},
	\begin{align*}
	\| \Mr b_k \|_{L^{\pp} (\R^n)}^{p_-}
		&\leq \| \Mr b_k \chi_{Q_k^*} \|_{\pp}^{p_-} + 
\| \Mr b_k \chi_{^c Q_k^*} \|_{\pp}^{p_-} = I_1 + I_2.
\end{align*}
By \eqref{CZ-bk} we immediately get $I_1 \leq C\| \Mg f \chi_{Q_k^*}
\|_{\pp}^{p_-} = C\|M_N f\|_{\Lp(Q_k^*)}^{p_-}$.

To estimate $I_2$, let $B_0$ be the ball centered at
$x$ with radius $c_n |x - x_k|$, with $c_n$ a dimensional constant
such that $Q_k^{\ast} \subset B_0$. Then by the definition
of $M$,
\begin{equation} \label{eqn:maxest-far}
 M (\chi_{Q_k^*}) (x) \geq \frac{1}{|B_0|} \int_{B_0} \chi_{Q_k^*}
dx = \frac{|Q_k^*|}{|B_0|} \geq c(n) \frac{\ell_k^{n }}{|x - x_k|^{n }}.
\end{equation}
Therefore, by inequality~\eqref{CZ-bk2}, Lemma~\ref{lemma:homog-exp},
the boundedness of the maximal operator, and the fact that
$Q_k^*\subset E_\lambda$,
\begin{multline*}
I_2
= \| \Mr b_k \chi_{Q_k^*} \|_{\pp}^{p_-}
\leq C \lambda \left\| \frac{\ell_k^{n + 1}}{|x - x_k|^{n + 1}} 
\chi_{\R^n\setminus Q_k^*} \right\|_{\pp}^{p_-} 
\leq C \lambda \left\| M (\chi_{Q_k^*}) \chi_{\R^n\setminus Q_k^*} 
\right\|_{\pp \frac{n + 1}{n}}^{\frac{n + 1}{n}p_-} \\
\leq C \lambda \left\| \chi_{Q_k^*} \right\|_{\pp  \frac{n + 1}{n}}^{\frac{n + 1}{n}p_-} 
= C\lambda \| \chi_{Q_k^*} \|_{\pp}^{p_-} 
 \leq C \lambda \left\| \frac{\Mg f}{\lambda} \chi_{Q_k^*}
\right\|_{\pp}^{p_-} 
 = C \left\| \Mg f \chi_{Q_k^{\ast}} \right\|_{\pp}^{p_-}.
\end{multline*}

\bigskip

Now suppose $d\geq 1$;  We modify the above construction as follows. 
We have that the maximal operator is bounded on
$L^{\pp\frac{n+d+1}{n}}$.  Let
$\mathcal{H}_d$ be the space of polynomials of degree at most $d$,
considered as a subspace of the Hilbert space $L^2(Q_K^{\ast},
\tilde{\eta}_k dx)$.  Let  $c_k$ be the projection of $f$ onto
$\mathcal{H}_d$:  then for all $q\in \mathcal{H}_d$, 
$\langle f - c_k, q \eta_k \rangle = 0$.
We again define $b_k = (f - c_k) \eta_k$, $b=\sum b_k$, and $g=f-b$.
Then we have (see Stein~\cite[pp.~104-5]{stein93}) that $|g(x)|\leq
C\lambda$ and
\begin{align*}
	\Mr b_k (x) &\leq c \Mg f (x) \\	
\intertext{ if  $x \in Q_k^* $, and}
	\Mr b_k (x) &\leq c\lambda \frac{\ell_k^{n + d + 1}}{|x - x_k|^{n + d + 1}} 
\end{align*}
if $x\in \R^n\setminus Q_k^*$.  
We now repeat the argument above.   The estimate for $I_1$ is the same;
the estimate for $I_2$ is nearly so:
\begin{multline*}
	I_2 = \| \Mr b_k \cdot \chi_{Q_k^*} \|_{\pp}^{p_-}
	\leq C \lambda \left\| \frac{\ell_k^{n + d + 1}
          \chi_{^c Q_k^*}}{|x - x_n|^{n + d + 1}} \right\|_{\pp}^{p_-} \\
\leq C \lambda \left\| M (\chi_{Q_k^*}) \right\|_{\pp \frac{n + d + 1}{n}}^{\frac{n + d + 1}{n} p_-}
\leq C \lambda \left\| \chi_{Q_k^*} \right\|_{\pp  \frac{n + d +
    1}{n}}^{ \frac{n + d + 1}{n}p_-}
\leq C \left\| \Mg f \ \chi_{Q_k^{\ast}} \right\|_{\pp}^{p_-}.
\end{multline*}
This completes the proof of \eqref{Mr-Mg}.
\end{proof}

\section{The atomic decomposition:  $(\pp,\infty)$ atoms}
\label{section:p-infty-atoms}

We begin with the definition of atoms.   

\begin{definition}
    Given $\pp \in \MP$, and $q$, $1 < q \leq  \infty$, a
    function  $a(\cdot)$ is a $(\pp, q)$ atom if $\supp(a)\subset
    B=B(x_0,r)$ and it  satisfies
	\begin{enumerate}
	\item[(i)] $\| a \|_{q} \leq |B|^{\frac{1}{q}}\| \chi_B \|_{\pp}^{-1} $, \\
	\item[(ii)] $\ds \int a(x) x^{\alpha} dx = 0$ for all 
$\ds |\alpha| \leq \lfloor n ( p_0^{-1} - 1)\rfloor$.
	\end{enumerate}
In (i) we interpret $1/\infty=0$.   These two conditions are called the size and vanishing moments
conditions of atoms.
\end{definition}

\begin{remark}
If $p_0>1$ (which can happen if $p_->1$), then $\lfloor n ( p_0^{-1} -
1)\rfloor<0$, and we interpret this to mean that no vanishing moments
are required.
\end{remark}

\medskip

In the remainder of this section we consider the case $q=\infty$.

\begin{theorem}\label{Thm-Atomic}
 Suppose $\pp \in \MP$. Then a distribution $f$ is in $H^{\pp}(\R^n)$ if
  and only if there exists a collection $\{ a_j \}$ of $(\pp, \infty)$ atoms
  supported on balls $\{ B_j \}$, and non-negative coefficients $\{\lambda_j\}$
  such that
 \[ f = \sum_j \lambda_j a_j, \]
where the series converges in $H^\pp(\R^n)$.  Moreover,
\begin{equation} \label{eqn:atomic1}
\| f \|_{H^{\pp}} \simeq
\inf\left\{ \bigg\| \sum_j \lambda_j \frac{\chi_{B_j}}
{\| \chi_{B_j} \|_{\pp}} \bigg\|_{\pp} : f = \sum_j \lambda_j a_j \right\}.
\end{equation}
\end{theorem}

\begin{remark}
\label{cor-Atom-Norm} As an immediate corollary we get that $(\pp,\infty)$ atoms are
uniformly bounded in $H^\pp$.   However, as we will see, unlike the
classical case we will not use this fact to prove the  boundedness of
operators.
\end{remark}

\smallskip

Theorem~\ref{Thm-Atomic} follows from two lemmas whose proof
we defer momentarily.  

\begin{lemma} \label{lemma:atoms-in-Hp} Given $\pp \in \MP$, suppose
  $\{a_j\} $ is a sequence of $(\pp, \infty)$ atoms, supported on $B_j
  = B(x_j, r_j)$, and $\{\lambda_j\}$ is a non-negative sequence that satisfies
\begin{equation} \label{eqn:atomHp1}
 \bigg\| \sum_j \lambda_j 
\frac{\chi_{B_j}}{\| \chi_{B_j} \|_{\pp}} \bigg\|_{\pp} < \infty.
\end{equation}
Then the series $f = \sum_j \lambda_j a_j$ converges in
$H^{\pp}$, and 
 \begin{align} \label{Atom-Coeff1}
\| f \|_{H^{\pp}} \leq 
C (n, \pp, p_0)\bigg\| \sum_j \lambda_j \frac{\chi_{B_j}}{\| \chi_{B_j} \|_{\pp}} \bigg\|_{\pp}.
	\end{align}
\end{lemma}

\begin{lemma} \label{lemma:Atom-Stein} 
  Let $\pp \in \MP$. If $f \in H^{\pp}$, then there exist
  $(\pp,\infty)$ atoms
  $\{ a_{k, j} \}$, supported on balls $B_{k, j}$, and non-negative coefficients
  $\{ \lambda_{k, j} \}$ such that
	\begin{align}
	\label{Atom-Decomp}
	f = \sum_{k, j} \lambda_{k, j} a_{k, j}.
	\end{align}
Moreover,
 \begin{align} \label{Atom-Coeff2}
\bigg\| \sum_{k, j} \lambda_{k, j} \frac{\chi_{B_{k, j}}}
{\| \chi_{B_{k, j}} \|_{\pp}} \bigg\|_{\pp} \leq C (n, \pp, p_0)\| f \|_{H^{\pp}}.
	\end{align}
\end{lemma}

\begin{proof}[Proof of Theorem \ref{Thm-Atomic}] 
  By Lemmas~\ref{lemma:atoms-in-Hp} and~\ref{lemma:Atom-Stein}, $f\in
  \Hp$ if and only if it has the desired atomic decomposition.
  Therefore, it remains to show that \eqref{eqn:atomic1} holds.  
  Given $f \in H^{\pp}$, there exists an atomic decomposition such
  that \eqref{Atom-Coeff2} holds.  This shows that the $H^\pp$ norm of
  $f$ dominates the infimum of the atomic decomposition norms.  To see
  the opposite inequality, given any decomposition $f = \sum_j
  \lambda_j a_j$, \eqref{Atom-Coeff1} holds. Since this is true for
  all atomic decomposition, we have that $\| f \|_{H^{\pp}}$ is  majorized by
  the infimum of the atomic decomposition norms.
	\end{proof}

        Throughout the rest of this section, let $d = \lfloor
        n(\frac{1}{p_0} - 1) \rfloor$ and $\gamma = (n + d + 1)/n$.
        Since $\pp \in \MP$, $M$ is also bounded on $L^{\gamma\pp }$.
        For by definition, $d > n(\frac{1}{p_0} - 1) - 1$, and
        this is equivalent to $\frac{n + d + 1}{n} > \frac{1}{p_0}$.
        Thus by Lemma~\ref{lemma:max-up} we get the boundedness of
        $M$.

\begin{proof}[Proof of Lemma~\ref{lemma:atoms-in-Hp}]
  Fix $\Phi\in \Ss$ such that $\int \Phi\,dx \neq 0$ and
  $\supp(\Phi)\subset B(0,1)$.  Fix atoms $\{ a_j \}$ with support $\{
  B_j \}$ and coefficients $\{\lambda_j\}$ such
  that~\eqref{eqn:atomHp1} holds.  Given $B = B(x_0, r)$, let $2B =
  B(x_0, 2r)$. We consider the case $p_-<1$; if $p_-\geq 1$ the proof
  is essentially the same, omitting the exponent $p_-$.  By
  Lemma~\ref{lemma:minkowski-low},
 \begin{multline*}
\| \Mr f \|_{\pp}^{p_-} \lesssim \| \sum_{j} \lambda_j \Mr (a_j) \|_{\pp}^{p_-} \\
\leq \underbrace{\bigg\| \sum_{j} \lambda_j \Mr (a_j) 
\cdot \chi_{2B_j} \bigg\|_{\pp}^{p_-}}_{I_1}
 + \underbrace{\bigg\| \sum_{j} \lambda_j \Mr (a_j) 
\cdot \chi_{(2B_j)^c}\bigg \|_{\pp}^{p_-}}_{I_2}.
	\end{multline*}
We first estimate $I_1$.  By the size condition on $(\pp,\infty)$
atoms, we have that 
\begin{equation} \label{eqn:atom-size}
 \Mr a_j (x) \leq \|a_j\|_\infty \|\Phi\|_1 \leq c\| \chi_{B_j} \|_{\pp}^{-1}.
\end{equation}
 Define $g_j = (\| \chi_{B_j} \|_{\pp}^{-1} \lambda_j)^{p_0}
 \chi_{B_j}$. If $x \in \chi_{2B_j}$, then by the definition of the
 maximal operator, 
\[ 
	M g_j (x) \geq (\| \chi_{B_j} \|_{\pp}^{-1}
        \lambda_j)^{p_0} \frac{1}{|2B_j|} \int_{2B_j} \chi_{B_j} dx 
= 2^{-n}(\| \chi_{B_j} \|_{\pp}^{-1}  \lambda_j)^{p_0}.
\]
Then  by Lemmas~\ref{lemma:homog-exp} and~\ref{lemma:max-vector},
\begin{multline*}
I_1 	
\leq C\bigg\| \sum_j \| \chi_{B_j} \|_{\pp}^{-1}  \lambda_j 
\ \chi_{2B_j} \bigg\|_{\pp}^{p_-}  \leq C \bigg\| \sum_j M
(g_j)^{1/p_0} 
\bigg\|_{\pp}^{p_-} \\
= C \bigg\| \bigg( \sum_j M (g_j)^{1/p_0} \bigg)^{p_0}
\bigg\|_{\frac{\pp}{p_0}}^{\frac{p_-}{p_0}} 
\leq C  \bigg\| \bigg(
\sum_j g_j^{1/p_0} \bigg)^{p_0}
\bigg\|_{\frac{\pp}{p_0}}^{\frac{p_-}{p_0}}  
= C \bigg\| \sum_j\| \chi_{B_j} \|_{\pp}^{-1} 
\lambda_j \chi_{B_j} \bigg\|_{\pp}^{p_-}.
\end{multline*}

To estimate $I_2$,  let $a$ be an atom supported on $B=B(x_0,r)$.
Then arguing as in~\cite[p.~106]{stein93} we have for $x \in (2B)^c$
the pointwise estimate
\begin{multline}	\label{Atom-pw}
\Mr a(x) \leq c \left( \frac{r}{|x - x_0|} \right)^{n + 1 + d}
\avgint_B a(y)\,dy \\
\leq \left( \frac{r}{|x - x_0|} \right)^{n + 1 + d} \|a\|_\infty
\leq  c \left( \frac{r}{|x - x_0|} \right)^{n \gamma} \|\chi_{B} \|_{\pp}^{-1}.  
\end{multline}
Now arguing as we did in
the proof inequality~\eqref{eqn:maxest-far}, we have for each $j$ that 
\begin{equation} \label{eqn:atom-pw2}
\Mr (a_j) (x)\leq  c \left( \frac{r_j}{|x - x_j|} \right)^{n\gamma} \| \chi_{B_j} \|_{\pp}^{-1}  
\leq c \| \chi_{B_j} \|_{\pp}^{-1}  M (\chi_{B_j})^{\gamma}.
\end{equation}
We can now estimate $I_2$: by Lemmas~\ref{lemma:homog-exp}
and~\ref{lemma:max-vector}, 
\begin{multline*}
I_2 \leq \bigg\| \sum_j \lambda_j \Mr (a_j) 
\cdot \chi_{(2B_j)^c} \bigg\|_{\pp}^{p_-} \\
\leq c \bigg\| \sum_j \frac{\lambda_j}{\| \chi_{B_j} \|_{\pp}}  
M (\chi_{B_j})^{\gamma} \bigg\|_{\pp}^{p_-} 
= \bigg\| \bigg( \sum_j M \bigg( \frac{\lambda_j^{1/\gamma}}
{\| \chi_{B_j} \|_{\pp}^{1/\gamma}} 
\chi_{B_j} \bigg) ^{\gamma} \bigg)^{1/\gamma} \bigg\|_{\gamma \pp}^{\gamma p_-} \\
\leq C  \bigg\| \bigg( \sum_j \| \chi_{B_j} \|_{\pp}^{-1}
\lambda_j\chi_{B_j} \bigg)^{1/\gamma} \bigg\|_{\gamma \pp}^{\gamma p_-} 
= C \bigg\| \sum_j \| \chi_{B_j} \|_{\pp}^{-1}  \lambda_j \chi_{B_j} \bigg\|_{\pp}^{p_-}.
\end{multline*}
\end{proof}

\begin{proof}[Proof of Lemma~\ref{lemma:Atom-Stein}]
We will prove this result assuming $f \in H^{\pp} \cap
 L_{loc}^1$; then by Proposition~\ref{Prop-Density} and a density
 argument (cf.~\cite[p.~109]{stein93}) we get it for arbitrary $f\in H^\pp$. 

 Fix such an $f$ and for every $j \in \Z$, let $E_{j} = \{ x : \Mg f (x)
 > 2^j \}$.  By Theorem~\ref{Thm-CZ} we have that 
 $f = g^{j} + b^{j}$,  where $|g^{j} (x)| \leq c 2^j$ and
 $b^{j} = \sum_k b_k^{j}$, with each $b_k^{j}$ supported on
a  cube $Q_{k}^{j \ast}$.  These cubes have bounded overlap
$E_j = \bigcup_k Q_{k}^{j   \ast}$.   Moreover, we have that
\begin{equation} \label{bj-0}
\lim_{j \rightarrow \infty} \| b^{j} \|_{H^{\pp}} = 0.
\end{equation}
To show this we proceed as in the proof of Lemma~\ref{lemma:atoms-in-Hp} (again
only considering the case $p_-<1$):
\[ 
\| b^{j} \|_{H^{\pp}}^{p_-}\leq 
\underbrace{\bigg\| \sum_k \Mr(b_k^{j}) 
\cdot \chi_{Q_{k}^{j \ast}}\bigg \|_{\pp}^{p_-}}_{I_1}
 + \underbrace{\bigg\| \sum_k \Mr (b_k^{j}) \cdot \chi_{(Q_{k}^{j
       \ast})^c} \bigg\|_{\pp}^{p_-}}_{I_2}.
\]
We first estimate $I_1$:  by  \eqref{CZ-bk} we have that
\[ 	I_1	\leq c\| \sum_k \Mg f \cdot \chi_{Q_{k}^{j \ast}}
\|_{\pp}^{p_-}  
\leq c \| \Mg f \cdot \chi_{E_j} \|_{\pp}^{p_-}. \]
The last term tends to $0$ as $j \rightarrow 0$:  this follows by
Lemma~\ref{lemma:norm-mod} and the dominated convergence theorem. 

To estimate $I_2$, let $x_{k, j}$ and $\ell_{k, j}$ be the center and
side length of $Q_k^{j \ast}$.  Then arguing as we did for
inequality~\eqref{eqn:maxest-far}, if $x\in (Q_{k}^{j \ast})^c$, then
\[ 	M (\chi_{Q_k^{j \ast}}) (x) \geq c\frac{\ell_{k, j}^n}{|x -
  x_k|^n}. \]
Then by inequality~\eqref{CZ-bk2} and Lemma~\ref{lemma:max-vector}, 
\begin{multline*}
	I_2 	
= \bigg\| \sum_k \Mr (b_k^{j}) \cdot \chi_{Q_k^{j \ast}}
\bigg\|_{\pp}^{p_-} 
\leq c \bigg\| \sum_k 2^j \frac{\ell_{k, j}^{n + 1 + d}}{|x - x_k|^{n + d + 1}} \cdot \chi_{Q_k^{j \ast}} \bigg\|_{\pp}^{p_-} \\
\leq c 2^{j p_-} \bigg\| \sum_k  M (\chi_{Q_k^{j \ast}})^{\gamma}
\bigg\|_{\pp}^{p_-} 
= c 2^{j p_-} \bigg\| \bigg( \sum_k (M \chi_{Q_k^{j \ast}})^{\gamma} \bigg)^{1/\gamma} \bigg\|_{\gamma \pp}^{\gamma p_-}  \\
\leq c 2^{jp_-} \bigg\| \bigg( \sum_k (\chi_{Q_k^{j \ast}})^{\gamma}
\bigg)^{1/\gamma} \bigg\|_{\gamma \pp}^{\gamma p_-}  
= c \bigg\| \sum_k 2^j \chi_{Q_k^{j \ast}} \bigg\|_{\pp}^{p_-} \\
\leq c\| 2^j \chi_{\{ x : \Mg f (x) > 2^j \}} \|_{\pp}^{p_-} \leq c\| \Mg f \chi_{E_j} \|_{\pp}^{p_-}.
\end{multline*}
As before, the last term goes to $0$ as $j\rightarrow \infty$. This
proves the limit~\eqref{bj-0}.                

\medskip

As a consequence of~\eqref{bj-0} we have that $g_j\rightarrow f$ in
norm (and so in $\Ss'$) as $j\rightarrow \infty$.  Further, since
$g_j\rightarrow 0$ uniformly as $j\rightarrow -\infty$, we have that
\[ f = \sum_j (g^{j + 1} - g^{j}).\]
From the proof of Theorem~\ref{Thm-CZ}, let $\{\eta_k^j\}$ be the
partition of unity for $E_j$ with $\supp(\eta_k^j)\subset Q_k^{j*}$.
Since $g^{j+1}-g^j = b^{j+1}-b^j$, $\supp(g^{j+1}-g^j)\subset
E_j$. Therefore, we have that
\[ f = \sum_{j,k} (g^{j + 1} - g^{j})\eta_k^j.\]

We now want to show that this expression can be rewritten as sum of
atoms.  Our argument follows Stein~\cite[pp.108--9]{stein93}, and
since many details are the same, we omit them here.
Again as in the proof of Theorem~\ref{Thm-CZ}, define the projections
$\mathcal{P}_k^{j} : \mathcal{S}' \rightarrow \mathcal{H}_d$, where
$\mathcal{H}_d$ is the space of polynomials of degree at most $d$,
thought of as a subspace of the Hilbert space
$L^2(Q_k^{j*},\tilde{\eta}_k^k\,dx)$. Define the polynomials $c_k^j =
\mathcal{P}_k^j (f)$  and $c_{\ell}^{j + 1} = \mathcal{P}_{\ell}^{j +
  1} (f)$, and let $c_{k, \ell} = \mathcal{P}_{\ell}^{j + 1} [(f -
c_{\ell}^{j + 1}) \eta_k^j]$. 
For each $j$, we can then write
	\[ g^{(j + 1)} - g^{j} = b^{j} - b^{(j + 1)} = \sum_k (f - c_k^j) \eta_k^j - \sum_{\ell} (f - c_{\ell}^{j + 1}) \eta_{\ell}^{j + 1} = \sum_k A_k^j, \]
where 
		\[ A_k^j = (f - c_k^j) \eta_k^j - \left( \sum_{\ell} (f - c_{\ell}^{j + 1}) \eta_{\ell}^{j + 1} \right) \eta_k^j + \sum_{\ell} c_{k, \ell} \eta_{\ell}^{j + 1}. \]

There exists a ball $B_{k, j}=B(x_{k,j},c\ell_{k,j})$ containing
the cube $Q_k^{j \ast}$ such that $|B_{k,j}|\leq c|Q_k^{j*}|$.
Moreover we have that $|A_k^j|\leq c2^j$ and $A_k^j$ satisfies the moment
conditions for $(\pp,\infty)$ atoms.  Therefore, if we define
\begin{equation} \label{eqn:patch2}
 a_{k, j} =  A_k^j c^{-1}2^{-j} \| \chi_{B_{k, j}} \|_{\pp}^{-1}, \qquad
	\lambda_{k, j} = c 2^j \| \chi_{B_{k, j}} \|_{\pp},
      \end{equation}
the $a_{k,j}$ are $(\pp,\infty)$ atoms and we have the decomposition
\eqref{Atom-Decomp}.  It converges in $\Ss'$, and so, arguing as in
the proof of Proposition~\ref{Prop-Density}, it converges in
$H^\pp(\R^n)$.

Finally, we prove~\eqref{Atom-Coeff2}.  We consider the case $p_-<1$; if $p_-\geq
1$, modify the following argument by replacing $1/p_0$ by $q>1$.  
Since $|B_{k,j}|\leq c|Q_k^{j*}|$, $M(\chi_{Q_k^{j*}}) \geq
c\chi_{B_{k,j}}=c\chi_{B_{k,j}}^{p_0}$.  Therefore, by
Lemmas~\ref{lemma:homog-exp} and~\ref{lemma:max-vector}, 
\begin{multline*}
\bigg\| \sum_{k, j} \frac{\lambda_{k, j}}{\| \chi_{B_{k, j}} \|_{\pp}} 
\ \chi_{B_{k, j}} \bigg\|_{\pp}
\leq C \bigg\| \sum_{k, j}\big(2^{jp_0}  M (\chi_{Q_k^{j*}})\big)^{1/p_0}  \bigg\|_{\pp}\\
 = C \bigg\| \bigg(\sum_{k, j} M (2^{jp_0}\chi_{Q_k^{j*}})^{1/p_0}
\bigg)^{p_0} \bigg\|_{\pp/p_0}^{1/p_0}
 \leq C \bigg\| \bigg(\sum_{k, j} 2^j \chi_{Q_k^{j\ast}}
 \bigg)^{p_0}\bigg\|_{\pp/p_0}^{1/p_0}  \\
 = C \bigg\| \sum_{k, j} 2^j \chi_{Q_k^{j\ast}} \bigg\|_{\pp}
\leq C \bigg\| \sum_{ j} 2^j \chi_{E_j} \bigg\|_{\pp}.
\end{multline*}
If $x \in \R^n$, there exists a unique $j_0 \in \Z$ such that $2^{j_0}
< \Mg f (x) \leq 2^{j_0 + 1}$. Hence, 
\[ \sum_j 2^j \chi_{E_j} (x) = \sum_{j \leq j_0} 2^j = 2^{j_0 + 1} \leq 2 \Mg f (x). \]
If we combine this with the previous estimate, we get \eqref{Atom-Coeff2}.

\end{proof}

\section{The atomic decomposition:  $(\pp,q)$ atoms}
\label{section:pq-atoms}

In this section we consider the atomic decomposition
when $q<\infty$. Our first main result is that when $q$ is sufficiently
large, the analog of Theorem~\ref{Thm-Atomic} holds. 
Furthermore,  we show that in this case
we can give a \textit{finite} atomic decomposition of
$H^{\pp}$ (Theorem~\ref{Var-Thm} below). Lastly, by minor modifications
to the proof of Theorem~\ref{Var-Thm}, we give a finite atomic
decomposition of the weighted Hardy space $H^{p_0} (w)$ 
(Theorem~\ref{lemma:Weighted-FinDecomp} below).  We use this to prove
the boundedness of singular integral operators on $H^{\pp}$ in Section 8. 

\subsection{Infinite atomic decomposition using $(\pp, q)$ atoms} 
We extend Theorem \ref{Thm-Atomic} by giving an atomic
decomposition using $(\pp, q)$ atoms.

\begin{theorem} \label{thm:qatoms}
Suppose $\pp \in \MP$.  Then a distribution $f$ is in $H^\pp$ if
and only if for $q>1$ sufficiently large, there exists a collection
$\{a_j\}$ of $(\pp,q)$ atoms supported on balls $\{B_j\}$, and
non-negative coefficients $\{\lambda_j\}$ such that
\[ f =\sum_j \lambda_j a_j, \]
where the series converges in $H^\pp(\R^n)$.  Moreover,
\begin{equation} \label{eqn:qatom1}
\| f \|_{H^{\pp}} \simeq
\inf\left\{ \bigg\| \sum_j \lambda_j \frac{\chi_{B_j}}
{\| \chi_{B_j} \|_{\pp}} \bigg\|_{\pp} : f = \sum_j \lambda_j a_j \right\}.
\end{equation}
\end{theorem}

\begin{remark} Denote the norm of the maximal operator by
$\|M\|_{(\pp/p_0)'}$. Then it suffices to take $q> \max(1,p_+,p_0(1+2^{n+3}\|M\|_{(\pp/p_0)'}))$. 
\end{remark}

One half of the proof of  Theorem~\ref{thm:qatoms} is immediate: since for any $q$,
$1<q<\infty$, $|B|^{1/q}\|a\|_q \leq \|a\|_\infty$, $(\pp,\infty)$
atoms are $(\pp,q)$ atoms.  Therefore, by
Lemma~\ref{lemma:Atom-Stein}, every function $f\in \Hp$ can be
written as the sum of $(\pp,q)$ atoms and $\|f\|_{H^\pp}$ has
the desired bound.   Note that in this case there are no restrictions
on $q$.   The heart of the proof, therefore, is to prove the converse.

\begin{lemma} \label{lemma:q-atoms-in-Hp} Given $\pp \in \MP$, there
  exists $q=q(\pp,p_0,n)>\max(p_+,1)$ such that if 
  $\{a_j\} $ is a sequence of $(\pp, q)$ atoms supported on $B_j
  = B(x_j, r_j)$, and $\{\lambda_j\}$ is a non-negative sequence that satisfies
\begin{equation} \label{eqn:qatomHp1}
 \bigg\| \sum_j \lambda_j 
\frac{\chi_{B_j}}{\| \chi_{B_j} \|_{\pp}} \bigg\|_{\pp} < \infty,
\end{equation}
then the series $f = \sum_j \lambda_j a_j$ converges in
$H^{\pp}(\R^n)$, and 
 \begin{align} \label{eqn:qatomHp2}
\| f \|_{H^{\pp}} \leq 
C (n, \pp, p_0,q)\bigg\| \sum_j \lambda_j \frac{\chi_{B_j}}{\| \chi_{B_j} \|_{\pp}} \bigg\|_{\pp}.
	\end{align}
\end{lemma}

To prove Lemma~\ref{lemma:q-atoms-in-Hp} we will adapt the proof 
of Rubio de Francia extrapolation in the setting of
variable Lebesgue spaces.  This was first proved
in~\cite{MR2210118} (see
also~\cite{cruz-fiorenza-book,cruz-martell-perezBook}).   We need more
careful control of the constants than was given in
the original proof, and so we will reproduce the key steps.

To apply extrapolation we need a version of
Lemma~\ref{lemma:q-atoms-in-Hp} for weighted $H^p$ spaces.  To state
it we introduce some definitions and preliminary results.
For complete information on the theory of weights,
see~\cite{duoandikoetxea01,garcia-cuerva-rubiodefrancia85}.
By a weight $w$ we will always mean a non-negative, locally integrable
function that is positive almost everywhere.  We will say that $w\in
A_1$ if 
\[ [w]_{A_1} = \esssup_{x\in \R^n} \frac{Mw(x)}{w(x)} < \infty. \]
Equivalently, $w\in A_1$ if  given any ball $B$,
\[ \avgint_B w(y) \,dy \leq [w]_{A_1} \essinf_{x\in B} w(x). \]
A
weight satisfies the reverse H\"older inequality with exponent $s>1$,
denoted by $w\in RH_s$, if for every cube $B$,
\[ \left(\avgint_B w(x)^s\,dx\right)^{1/s} \leq C\avgint_Q
w(x)\,dx; \]
the best possible constant is denoted by $[w]_{RH_s}$.  Note that if
$w\in RH_s$, then by H\"older's inequality, $w\in RH_t$ for all $t$,
$1<t<s$, and $[w]_{RH_t}\leq [w]_{RH_s}$.  If $w\in A_1$, then $w\in
RH_s$, and we have sharp control over the exponent $s$. 

\begin{lemma} \label{lemma:sharp-RH}
Given $w\in A_1$, then $w\in RH_s$, where
$s=1+\big(2^{n+2}[w]_{A_1})^{-1}$.  
\end{lemma}

\begin{remark} \label{remark:rh-const}
  This result is proved in~\cite{MR2427454} (see
  also~\cite{cruz-fiorenza-book}), where everything is done in terms
  of cubes instead of balls.  However, because $w\in A_1$ is doubling,
  the reverse H\"older inequality holds for balls with same exponent;
  in this case the constant $[w]_{RH_s}$ depends on
  $n$ and~$[w]_{A_1}$.
\end{remark}

Given a weight $w\in A_1$ and $p_0>0$, the weighted Hardy space
$H^{p_0}(w)$ consists of all tempered distributions $f$ such that 
\[ \|f\|_{H^{p_0}(w)}=  \|\Mr f\|_{L^{p_0}(w)} = \left(\int_\subRn \Mr f(x)^{p_0}
  w(x)\,dx\right)^{1/p_0} < \infty.  \]
These spaces have an atomic decomposition:  see Str\"omberg and
Torchinsky~\cite{MR1011673}.    We state their result in the form we need to
apply it; see Remark \ref{remark:atom-switch} below.

\begin{lemma} \label{lemma:ST-atoms} Given  $\pp \in \MP$ and
  $q>\max(p_0,1)$, suppose $\{a_j\}$ is a sequence of $(\pp,q)$
atoms, $\{\lambda_j\}$ is a non-negative sequence, and  $w\in A_1 \cap
  RH_{(q/p_0)'}$.  If 
\[ \bigg\| \sum_j \lambda_j \frac{\chi_{B_j}}{\|\chi_{B_j}\|_\pp}\bigg\|_{L^{p_0}(w)} <
\infty. \]
Then the series
\[ f = \sum_j \lambda_j a_j \]
converges in $H^{p_0}(w)$ and 
\[ \|f\|_{H^{p_0}(w)} \leq C (\pp,p_0,q,n, [w]_{A_1},[w]_{RH_{(q/p_0)'}})\bigg\| \sum_j \lambda_j
\frac{\chi_{B_j}}{\|\chi_{B_j}\|_\pp}\bigg\|_{L^{p_0}(w)}.\]
\end{lemma}

\begin{remark} \label{remark:atom-switch}
In~\cite[Chapter VIII, Theorem 1]{MR1011673} this result is stated for
atoms $\bar{a}_j$ that (obviously) do not depend on a variable
exponent $\pp$.  To pass between the two kinds of atoms, it suffices
to take $\bar{a}_j=\|\chi_{B_j}\|_\pp a_j$ and
$\bar{\lambda}_j = \lambda_j\|\chi_{B_j}\|_\pp^{-1}$.  
The atoms $\bar{a}_j$ are required to have vanishing moments for
$|\alpha| \leq  \lfloor
d/p-n \rfloor$, where $d$ is a constant such that for all $t\geq 1$, 
\[ w(B(x,tr)) \leq Kt^d w(B(x,r)). \]
If $w\in A_1$, then this is true with $d=n$:
\begin{multline*}
 w(B(x,tr)) \leq [w]_{A_1}|B(x,tr)| \essinf_{y\in B(x,tr)} w(y) \\
\leq [w]_{A_1} t^n |B(x,r)| \essinf_{y\in B(x,r)} w(y)
\leq [w]_{A_1} t^n w(B(x,r)).
\end{multline*}
\end{remark}

\begin{proof}[Proof of Lemma~\ref{lemma:q-atoms-in-Hp}]
Fix $\pp \in \MP$;  by Lemma~\ref{lemma:diening} the maximal operator is bounded
on $L^{(\pp/p_0)'}(\R^n)$.  Denote the norm of the maximal operator by
$\|M\|_{(\pp/p_0)'}$.  Fix $q> \max(1,p_+,p_0(1+2^{n+3}\|M\|_{(\pp/p_0)'}))$; the reason for
this choice will be made clear below.  

We will first show that if $a$ is a $(\pp,q)$ atom with support $B$,
then $a\in H^\pp$.  To do so we will show that $\|\Mr a\|_\pp < \infty$.    By
Lemma~\ref{lemma:minkowski-low} (if $p_-<1$; the case $p_-\geq 1$ is
handled similarly),
\[ \|\Mr a\|_\pp^{p_-} \leq \|\Mr (a)\chi_{2B}\|_\pp^{p_-}+
\|\Mr( a)\chi_{(2B)^c}\|_\pp^{p_-} = I_1 + I_2. \]
By Lemma~\ref{lemma:imbed}, since $q> \max(p_+,1)$ and $\Mr$ is bounded on
$L^q$, 
\[ I_1= \|\Mr (a)\chi_{2B}\|_\pp \leq (1+|2B|)\|\Mr (a) \chi_{2B}\|_q
\leq C(1+|2B|)\|a\|_q < \infty. \]

To show that $I_2$ is finite, by inequality~\eqref{Atom-pw}
and the definition of $(\pp,q)$ atoms, and arguing as we did for~\eqref{eqn:atom-pw2},
for $x\in (2B)^c$ we have that
\begin{multline*}
\Mr a(x) \leq c \left(\frac{r}{|x-x_0|}\right)^{n\gamma}
|B|^{-1/q}\|a\|_q  \\
\leq \left(\frac{r}{|x-x_0|}\right)^{n\gamma}
\|\chi_B\|_\pp^{-1} \leq c \|\chi_B\|_\pp^{-1} M(\chi_B)(x)^\gamma,
\end{multline*}
where $x_0$ is the center of $B$ and $\gamma= (n+d+1)/n$.  As we noted
in the proof of Theorem~\ref{Thm-Atomic}, $M$ is bounded on
$L^{\gamma\pp}$.  Therefore, by Lemma~\ref{lemma:homog-exp},
\[ I_2= \|\Mr a\chi_{(2B)^c}\|_\pp \leq 
c\|\chi_B\|_\pp^{-1} \|M(\chi_B)\|_{\gamma \pp}^\gamma
\leq c\|\chi_B\|_\pp^{-1} \|\chi_B\|_{\gamma \pp}^\gamma < \infty. \]

\medskip

To construct our weight $w$,  form the Rubio de Francia iteration
algorithm with respect to $L^{(\pp/p_0)'}$.   Given a function $h$,
define
\[ \mathcal{R} h  = \sum_{i=0}^\infty \frac{M^i f}{2^i \|M\|_{(\pp/p_0)'}}, \]
where $M^0 h = |h|$ and for $i\geq 1$, $M^i h = M\circ \cdots \circ M
h$ is $i$ iterates of the maximal operator.  Three facts follow at
once from this definition (cf.~\cite{MR2210118,cruz-martell-perezBook}):

\begin{enumerate}

\item $|h| \leq \mathcal{R} h$;

\item $\mathcal{R}$ is bounded on $L^{\pp/p_0)'}(\mathcal{R}^n)$ and $\|\mathcal{R}
    h\|_{(\pp/p_0)'} \leq 2\|h\|_{(\pp/p_0)'}$;

\item $\mathcal{R} h \in A_1$ and $[\mathcal{R} h]_{A_1} \leq
  2\|M\|_{(\pp/p_0)'} = C(\pp,p_0,n)$.  

\end{enumerate}
By Lemma~\ref{lemma:sharp-RH} we have that $\mathcal{R} h\in RH_s$,
where $s= 1+(2^{n+3}\|M\|_{(\pp/p_0)'})^{-1}$.  Therefore, since
$q \geq p_0(1+2^{n+3}\|M\|_{(\pp/p_0)'})$, we have that $\mathcal{R} h \in
RH_{(q/p_0)'}$ and $[\mathcal{R} h]_{Rh_{(q/p_0)'}}\leq C(\pp,p_0,n)$.
We stress that all of these constants are independent of $h$.

Fix a sequence of atoms $\{a_j\}$ and constants $\{\lambda_j\}$ as in
the hypotheses.  Let $f =\sum \lambda_ja_j$; {\em a priori} we do not
know that this series converges in $H^\pp$.  To avoid this problem, define the functions
\[ f_k  = \sum_{j=1}^k \lambda_j a_j. \]
Then  $f_k \in H^\pp(\R^n)$:  since $a_j\in H^\pp$,  by
Lemma~\ref{lemma:minkowski-low} (if $p_-<1$, the case $p_-\geq 1$ is
handled similarly)
\[ \|\Mr f_k \|_\pp^{p_-} \leq \sum_{j=1}^k \lambda_j^{p_-}\|\Mr a_j\|_\pp^{p_-}<\infty. \]
Furthermore, 
by Lemma~\ref{lemma:ST-atoms}, given any function $h$, $f_k \in
H^{p_0}(\mathcal{R} h)$, and 
\begin{equation} \label{eqn:ST-est}
 \|f\|_{H^{p_0}(\mathcal{R} h)} \leq 
C(\pp,p_0,n)\bigg\|\sum_{j=1}^k \lambda_j\frac{\chi_{B_j}}{\|\chi_{B_j}\|_\pp}
\bigg\|_{L^{p_0}(\mathcal{R}  h)}.
\end{equation}

We will now show that \eqref{eqn:qatomHp2} holds for each $f_k$ with a
constant independent of $k$.     By Lemmas~\ref{lemma:holder}
and~\ref{lemma:homog-exp}, 
\[ \|\Mr f_k \|_\pp^{p_0} = \|(\Mr f_k)^{p_0} \|_{\pp/p_0}
\leq C(\pp,p_0)\sup \int_\subRn \Mr f_k(x)^{p_0} h(x)\,dx, \]
where the supremum is taken over all $h\in L^{(\pp/p_0)'}$ with
$\|h\|_{(\pp/p_0)'}\leq 1$.  (We may assume that $h$ is non-negative.) Fix such a function $h$; we will estimate the integral on the
right-hand side with a constant independent of $h$.  By the properties
of the Rubio de Francia iteration algorithm,  \eqref{eqn:ST-est} and
Lemmas~\ref{lemma:holder} and~\ref{lemma:homog-exp},  
\begin{align*}
\int_\subRn \Mr f_k(x)^{p_0} h(x)\,dx
& \leq \int_\subRn \Mr f_k(x)^{p_0} \mathcal{R} h(x)\,dx \\
& \leq C\int_\subRn \bigg(\sum_{j=1}^k
\lambda_j\frac{\chi_{B_j}(x)}{\|\chi_{B_j}\|_\pp}\bigg)^{p_0} \mathcal{R} h(x)\,dx \\
& \leq C\bigg\|\bigg(\sum_{j=1}^k
\lambda_j\frac{\chi_{B_j}}{\|\chi_{B_j}\|_\pp}\bigg)^{p_0}\bigg\|_{\pp/p_0}\|\mathcal{R}
  h\|_{(\pp/p_0)'}  \\
& \leq C\bigg\|\sum_{j=1}^k
\lambda_j\frac{\chi_{B_j}}{\|\chi_{B_j}\|_\pp}\bigg\|_{\pp}^{p_0}.
\end{align*}
Inequality~\eqref{eqn:qatomHp2} for $f_k$ now follows and the constant depends only
on $\pp$, $p_0$ and $n$.  

To complete the proof we need to show that \eqref{eqn:qatomHp2} holds
for $f$.    But the same argument that proved this inequality for
$f_k$ shows that if $l>k$, 
\begin{equation} \label{eqn:cauchy}
 \|f_l -f_k\|_{H^\pp(\R^n)} \leq 
C\bigg\|\sum_{j=k+1}^l
\lambda_j\frac{\chi_{B_j}}{\|\chi_{B_j}\|_\pp}\bigg\|_{\pp}.
\end{equation}
However, by hypothesis we have that 
\[ \bigg\|\sum_{j=1}^\infty
\lambda_j\frac{\chi_{B_j}}{\|\chi_{B_j}\|_\pp}\bigg\|_{\pp}<\infty \]
and therefore the partial sums of this series are Cauchy in
$L^\pp(\R^n)$.  Hence, as $k,\,l \rightarrow \infty$, the right-hand
side of \eqref{eqn:cauchy} tends to $0$.  Therefore, the sequence
$\{f_k\}$ is Cauchy in $H^\pp$ and so by
Proposition~\ref{prop:complete} converges to $f$ in $H^\pp$.
Therefore, by the monotone convergence theorem in variable Lebesgue
spaces (Lemma~\ref{lemma:monotone}) we have that
\begin{multline*}
 \|f\|_{H^\pp(\R^n)} = \lim_{k\rightarrow \infty} \|f_k
 \|_{H^\pp(\R^n)} \\
\leq C \lim_{k\rightarrow \infty} 
\bigg\|\sum_{j=1}^k
\frac{\lambda_j}{\|\chi_{B_j}\|_\pp}\chi_{B_j}\bigg\|_{\pp}
=C \bigg\|\sum_{j}
\frac{\lambda_j}{\|\chi_{B_j}\|_\pp}\chi_{B_j}\bigg\|_{\pp}. 
\end{multline*}
\end{proof}

\bigskip

\subsection{Finite atomic decompositions}
Given $q<\infty$, let $H^{\pp,q}_{fin}$ be the  subspace of $H^\pp$ consisting of all $f$
that have decompositions  as finite sums of $(\pp,q)$ atoms.   By
Theorem~\ref{Thm-Atomic}, if $q$ is sufficiently large, $H^{\pp,q}_{fin}$ is dense
in $H^\pp$.      Our next result shows that on this subspace the
atomic decomposition norm, restricted to finite decompositions, is 
equivalent to the $H^{\pp}$ norm. 
This extends a result from \cite{MR2399059} to the variable
setting. 

\begin{theorem} \label{Var-Thm} Let $\pp \in \MP$ and fix $q$ as in
  Theorem~\ref{thm:qatoms}. For $f \in H_{fin}^{\pp, q} (\R^n)$,
  define
\begin{equation} \label{eqn:VT1}
 \|f\|_{H_{fin}^{\pp, q}} = \inf\bigg\{ \bigg\|\sum_{j=1}^k 
\lambda_j\frac{\chi_{B_j}}{\|\chi_{B_j}\|_\pp }\bigg\|_\pp 
: f = \sum_{j=1}^k \lambda_j a_j \bigg\},
\end{equation}
where infimum is taken over all finite decompositions of $f$ using $(\pp, q)$ atoms  $a_j$, supported on balls $B_j$.  Then
\[ \| f \|_{H^{\pp}}  \simeq \| f \|_{H_{fin}^{\pp, q}}. \]
\end{theorem}

Our argument is based on the proof of \cite[Theorem 6.2]{MR2492226}. 
It requires two lemmas.  The first
introduces a non-tangential variant of the grand maximal operator.  A
proof can be found in Bownik~\cite[Prop.~3.10]{MR1982689}.

\begin{lemma} \label{lemma:bownik-NT}
Define the non-tangential grand maximal function $\mathcal{M}_{N,1}$
by
\[ \mathcal{M}_{N,1} f(x) \sup_{\Phi \in \Ss_N} \sup_{|y-x|<t}
|\Phi_t*f(x)|. \]
Then for all $x\in \R^n$ and tempered distributions $f$,
\[  \mathcal{M}_{N,1} f(x) \approx \Mg f(x), \]
where the constants depend only on $N$.  
\end{lemma}

The second lemma is a decay estimate for the grand maximal operator.

\begin{lemma} \label{lemma:decay-Mg}
Given $\pp \in \MP$, suppose $f\in H^\pp$ is such that
$\supp(f)\subset B(0,R)$ for some $R>1$.  Then for all $x\in
B(0,4R)^c$,
\[ \Mg f(x) \leq C(N,\pp,p_0)\|\chi_{B(0,R)}\|_\pp^{-1}. \]
\end{lemma}

\begin{proof}
  To prove the desired estimate, it will suffice to show that for any
  $\Phi \in \mathcal{S}_N$, $x \in B(0, 4R)^c$, and $t > 0$,
\[ |f*\Phi_t(x)| \leq C\| \chi_{B(0, R)} \|_{\pp}^{-1}, \]
where the constant $C$ is independent of $f$, $\Phi$, $x$ and $t$.  We
consider two cases, depending on the size of $t$.

{\bf Case 1: $t \geq R$.}  Given $x \in B(0, 4R)^c$ and $t \geq R$, we
claim that there exists  $\Psi \in \mathcal{S}$ so that $f \ast \Phi_t (x) = f \ast
\Psi_R (0)$.  Let  $\theta \in C_c^{\infty}$ be such that
$\supp(\theta) \subset B(0, 2)$ and $\theta = 1$ on $B(0, 1)$,  and define $\Psi(z) =
\Phi(\frac{x}{t} + \frac{Rz}{t}) \theta(z) (R/t)^n$.  Then
\begin{multline*}
		f \ast \Phi_t (x) 	
	= \int f(y) t^{-n} \Phi\left(\frac{x - y}{t}\right) dy \\
	= \int f(y) t^{-n} \ \underbrace{t^{-n} \Phi \left(\frac{x - y}{t}\right)
          \theta\left(\frac{y}{R}\right)}_{\Psi_R(0 - y)} dy 
= f \ast \Psi_R (0).
\end{multline*}

We actually have that $c\Psi \in \mathcal{S}_N$, where
$c=c(\theta,N)$. To see this, recall that since $\Phi\in \Ss_N$,
$\|\partial^\beta \Phi\|_\infty \leq c$ for all $|\beta| \leq N$.  Fix $z \in \supp(\Psi) = B(0, 2)$.  Then
for any multi-index $|\beta| \leq N$,
\[ |\p^{\beta} \Psi(z)|
\leq \left( \frac{R}{t} \right)^{n } 
\sum_{\gamma \leq \beta} {\beta \choose \gamma} 
\bigg| \p^{\gamma} \Phi\left(\frac{x + Rz}{t}\right ) \left( \frac{R}{t} \right)^{\gamma }  
\p^{\beta - \gamma} \theta\left(\frac{y}{R}\right) R^{-|\beta| + |\gamma|} \bigg|.
\]
Since $t\geq R>1$, we see that $|\p^{\alpha} \psi(z)| \leq  C(\theta,N)$.  Hence,
\[ \sup_{|\alpha|,|\beta|\leq N} \| \Psi \|_{\alpha,\beta} 
= \sup_{|\alpha|,|\beta| \leq N} \sup_{z \in B(0, 2)} |z^\alpha \p^{\beta} \Psi(z)|
\leq C(\theta, N). \]

Since $c(N,\theta)\Psi \in \Ss_N$, by Lemma~\ref{lemma:bownik-NT} we
have the pointwise bound
\begin{multline*}
 |f \ast \Phi_t (x)|
= |f \ast \Psi_R (0)| 
\leq \inf_{z \in B(0, R)} M_{\Psi, 1} f(z)  \\
\leq C(N,\theta)\inf_{z \in B(0, R)} \mathcal{M}_{N,1} f(z) 
		\leq C(N,\theta) \inf_{z \in B(0, R)} \Mg f(z).
              \end{multline*}
Therefore, by Lemmas~\ref{lemma:holder}, \ref{lemma:homog-exp}
and~\ref{lemma:kopaliani}, 
\begin{multline*}
|f \ast \Phi_t (x)|^{p_0}
 \leq C \avgint_{B(0, R)} \Mg f(z)^{p_0} dz \\
 \leq C|B(0,R)|^{-1}\|(\Mg f)^{p_0}\|_{\pp/p_0}\|\chi_{B(0,R)}\|_{(\pp/p_0)'} \\
 \leq C\|\Mg f\|_\pp^{p_0}\|\chi_{B(0,R)}\|_{\pp/p_0}^{-1} 
 \leq C \| f \|_{H^{\pp}}^{p_0}  \| \chi_{B(0, R)} \|_{\pp}^{-p_0},
\end{multline*}
where $C=C(N,\pp,p_0,n)$.  

\smallskip

{\bf Case 2: $t < R$.}  This case is similar to the previous one but
we need to construct $\Psi$ differently as we need our estimate to
hold at more points than the origin.  Fix $z\in B(0,R/2)$ and choose
$u\in B(0,R/2)$ such that $|z-u|<t$. We claim there exists $\Psi$
(depending on $u$, $t$ and $R$) such that $f \ast \Phi_t (x) = f \ast
\Psi_t (u)$.  As before, let $\theta \in C_c^{\infty}$ be supported on
$B(0, 2)$ and $\theta = 1$ on $B(0, 1)$. Define $\Psi$ by
\[	\Psi(z) = \Phi\left( \frac{x - u}{t} + z \right) 
\theta\left( \frac{u}{R} - \frac{t}{R} z\right ). \]
Then we have that
\[ 	f \ast \Phi_t (x)	
= \int f(y) \Phi_t (x - y) dy 
= \int f(y) \underbrace{\Phi_t (x - y) \theta(y/R)}_{\Psi_t (u - y)} dy  
	= f \ast \Psi_t (u). \]
Assume for the moment that $c(\theta,N)\Psi \in \Ss_N$.  
Then by Lemma~\ref{lemma:bownik-NT},
\[ |f*\Psi_t(u)| \leq M_{\Psi,1} f(z) \leq C(\theta,N)
\mathcal{M}_{N,1}f(z) \leq C(\theta,N)\Mg f(z). \]
Since this holds for every $z\in B(0,R/2)$, we have that 
\[  |f \ast \Phi_t (x)| \leq C(\theta,N)\inf_{z\in B(0,R/2)} \Mg
f(z), \]
and we can repeat the above argument to get the desired estimate.  

\smallskip

It remains to show that $c(\theta,N)\Psi \in \mathcal{S}_N$; it will
suffice to show that for all $\beta$ such that $|\beta| \leq N$,
\[ \ds \sup_{z \in \R^n} |\p^{\beta} \psi(z)| (1 + |z|)^N \leq C(\theta,N). \]
Since $\Phi\in \Ss_N$, for all $|\beta|\leq N$,
$(1+|y|)^N|\partial^\beta \Phi(y)| \leq c(N)$.  Therefore, by the
product rule, since $t<R$, 
\begin{multline*} 
|\p^{\beta} \Psi(z)| 
\leq \sum_{\gamma\leq\beta}{\beta \choose \gamma}
\bigg| \partial^\beta \Phi\left(\frac{x-u}{t}+z\right)\bigg|
\left(\frac{t}{R}\right)^{|\beta|-|\gamma|}
\bigg|\partial^{\beta-\gamma}\theta\left(\frac{u}{R}-\frac{t}{R}z\right)\bigg|\\
\leq \frac{C(\theta,N)}{(1+|\frac{x-u}{t}+z|)^N}.
\end{multline*}

To estimate the last term, note first that since $x \not\in B(0, 4R)$ and $u \in B(0, R/2)$, 
    \[\frac{|x - u|}{t} > \frac{7}{2}\frac{R}{t}. \]
Second, since the  $\theta$ term is non-zero only if $|\frac{u -
  zt}{R}| \leq 2$, we must have that $|\frac{u}{t} - z| <
\frac{2R}{t}$, which implies $|z| < |\frac{u}{t}| + 2\frac{R}{t} =
\frac{5}{2}\frac{R}{t}$. Together these two estimates show that  
$|\frac{x - u}{t} + z| > \frac{7}{2}\frac{R}{t} -
\frac{5}{2}\frac{R}{t} 
= \frac{R}{t}$.  Therefore, for $z \in \supp(\Psi)$,
\[ 	|\p^{\beta} \psi(z)| (1 + |z|)^N
\leq C \frac{(1 + |z|)^N}{(1 + |\frac{x - u}{t} + z|)^N}  \leq C
\left( \frac{1 + 3R/t}{1 + R/t} \right)^N \leq C(\theta, N). 
\]
This completes the proof.
\end{proof}

\begin{proof}[Proof of Theorem~\ref{Var-Thm}]
Since the infimum over finite sums in \eqref{eqn:VT1} is larger than
the infimum when taken over all possible atomic decompositions, by
Theorem~\ref{thm:qatoms} we have that $\| f \|_{H^{\pp}}
\leq C\| f \|_{H_{fin}^{\pp, q}}$. 

To prove the reverse inequality,  fix $f \in H_{fin}^{\pp, q}$.   By
homogeneity we may assume that $\| f \|_{H^{\pp}} = 1$;    we will show
that $\| f \|_{H_{fin}^{\pp,  q}} \leq C(N,\pp,p_0,q,n)$.
Since $f$ has a finite atomic decomposition, 
there exists $R>1$ such that $\supp(f)\subset B(0,R)$.  
By Lemma~\ref{lemma:decay-Mg},  
\begin{equation}\label{Key-PW}
\Mg f (x) \leq c \| \chi_{B(0, R)} \|_{\pp}^{-1}.
\end{equation}
Let $\Omega_j = \{ x : \Mg f(x) > 2^j \}$; define $j' = j'(f,
\pp)$ to be the smallest integer such that for all $j> j'$,
$\overline{\Omega}_j \subset B(0, 4R)$. By
\eqref{Key-PW} it suffices to take $j'$ to be the largest integer such that
$2^{j'} < c \| \chi_{B(0, R)} \|_{\pp}^{-1}$.

By Lemma~\ref{lemma:Atom-Stein} we can form the ``canonical''
decomposition of $f$ in terms of $(\pp,\infty)$ atoms:
\[ f = \sum_{j} \sum_k \lambda_{k,j} a_{k,j} = \sum_{j \leq j'} \sum_k
\lambda_{k,j} a_{k,j} +
 \sum_{j> j'} \sum_k \lambda_{k,j} a_{k,j} = h + \ell. \]
We will rewrite the sum $h+\ell$ as a finite atomic
decomposition in terms of $(\pp,q)$ atoms.  To do so, we will
use the finer properties of the atoms $a_{k,j}$ that are implicit in
the proof of Lemma~\ref{lemma:Atom-Stein}.  First,
$\supp(h),\,\supp(\ell)\subset B(0,4R)$.  
The atoms $a_{k,j}$ are supported in $\Omega_j$, so  by
our choice of $j'$, $\supp(\ell)\subset B(0,4R)$.  Since
$\supp(f)\subset B(0,R)$, we also have that $\supp(h)\subset
B(0,4R)$.  

Second, $h,\,\ell \in L^q$.  Since $f$ has a finite $(\pp,q)$ atomic
decomposition it is in $L^q$; since $q>1$ we also have that $\Mg f\in
L^q$.  If we fix $x\in \supp(\ell)$, then there exists $s>j'$ such that
$x \in \Omega_{s} \backslash \Omega_{s + 1}$.  By construction
(see~\eqref{eqn:patch2}) the sets $\supp(a_{k,j})$ have bounded overlap and
$|\lambda_{k,j} a_{k,j}|\leq c2^j$.  Hence,
\begin{equation} \label{eqn:patch}
\sum_{j >j'} \sum_k |\lambda_{k,j} a_{k,j} (x)|
= \sum_{j'<j \leq s} \sum_k|\lambda_{k,j} a_{k,j} (x)| 
\leq c \sum_{j \leq s} 2^j  =c2^{s+1} \leq c \Mg f(x).
\end{equation}
Thus $\ell \in L^q$, and so $h = f - \ell \in L^q$ as well.

Third, $h,\,\ell$ satisfy the vanishing moment condition for all
$|\alpha|\leq \lfloor n(1/p_0-1) \rfloor$.   Since $f$ is a finite sum of $(\pp,q)$
atoms, it has vanishing moments for these $\alpha$. 
Since $\supp(\ell)\subset B(0,4R)$, by H\"older's inequality, $\ell
\in L^1$. Moreover, given any monomial $x^\alpha$, by \eqref{eqn:patch}
\[ \bigg\| \sum_{j> j'} \sum_k |x^\alpha||\lambda_{k,j} a_{k,j}| \bigg\|_{L^1} 
\leq (4R)^{|\alpha|}\bigg\| \sum_{j > j'} \sum_k |\lambda_{k,j} a_{k,j}|
\bigg\|_{L^q}
 \cdot |B(0, 4R)|^{1/q'} < \infty. \]
Thus the sum on the left-hand side converges absolutely in $L^1$ and
so we can exchange sum and integral to get that $\ell$ has the same
vanishing moments as each $a_{k,j}$.   Finally, since $h = f - \ell$,
$h$ also has the same vanishing moments.

Fourth, there exists a constant $c$ such that $ch$ is a $(\pp,
\infty)$ atom supported on $B(0, 4R)$. To show this we only need to
check the size condition. Fix $x \in \R^n$; then by the same estimates
for $a_{k,j}$ we used above, 
\[ 	|h(x)| 	\leq \sum_{j \leq j'} \sum_k |\lambda_{k,j}
a_{k,j}(x)| \leq c \sum_{j \leq j'} 2^j \leq c2^{j'} \leq c \| \chi_{B(0, R)} \|_{\pp}^{-1},
\]
where the last follows by our choice of $j'$. 

Finally, we show that $\ell$ can be rewritten as a finite sum of
$(\pp,q)$ atoms.  Let $F_i = \{ (j, k) : |j| + |k| \leq i \}$ and
define the finite sum $\ell_i$ by
\[ \ell_i= \sum_{F_i} \lambda_{k,j} a_{k,j}.  \]
Since the sum for $\ell$ converges
absolutely in $L^q$, we can find $i$ such that $\|\ell-\ell_i\|_q$ is
as small as desired.  In particular, we can find $i$ such that $\ell -
\ell_i$ is a $(\pp, q)$ atom.

Therefore, 
\[ f = c(h/c) + (\ell - \ell_i) + \sum_{(j, k) \in F_i} \lambda_{k,j} a_{k,j} \]
is a finite decomposition of $f$ as $(\pp,q)$ atoms.   To complete the
proof we will use Lemma~\ref{lemma:Atom-Stein} to get the desired 
estimate on $\|f\|_{H^{\pp,q}_{fin}}$.  
Let $\tilde{B} = B(0, 4R)$. By the definition of the finite atomic
norm and Lemma~\ref{lemma:minkowski-low} (if $p_-<1$),
\begin{multline*}
  \| f \|_{H_{fin}^{\pp, q}}^{p_-}
\leq \bigg\|  \frac{c \chi_{\tilde{B}}}{\| \chi_{\tilde{B}}
    \|_{\pp}^{p_-}} 
+  \frac{\chi_{\tilde{B}}}{\| \chi_{\tilde{B}} \|_{\pp}} 
+ \sum_{(j, k) \in F_i} \frac{\lambda_{k,j} \chi_{B_{k, j}}}{\|
  \chi_{B_{k, j}} \|_{\pp}}  \bigg\|_{\pp}^{p_-} 	 \\
 \leq c^{p_-}+1+\bigg\|\sum_{(j, k) \in F_i} \frac{\lambda_{k,j} \chi_{B_{k, j}}}{\|
  \chi_{B_{k, j}} \|_{\pp}}  \bigg\|_{\pp}^{p_-} 	 
 \leq C+\bigg\|\sum_{j, k} \frac{\lambda_{k,j} \chi_{B_{k, j}}}{\|
  \chi_{B_{k, j}} \|_{\pp}}  \bigg\|_{\pp}^{p_-} 	 
 \leq C+C\|f\|_{H^\pp(\R^n)}
 \leq C.
\end{multline*}
This completes the proof.  
\end{proof}

\subsection{Finite atomic decompositions for weighted Hardy spaces}
We end this section by showing that a version of Theorem~\ref{Var-Thm}
holds for the weighted Hardy spaces.    This result is of interest in its own
right, but we give it primarily because we will need it in the next
section to prove the boundedness of singular integrals on $H^\pp$.   For this
reason we only prove one particular case; we leave it to the
interested reader to prove the more general result implicit
in our work.  

Let $\pp \in \MP$, and let $q>1$.  Given $w\in A_1$, define 
$H_{fin}^{p_0, q} (w)$ to be the set of all finite sums of $(\pp,q)$
atoms.    By the proof of Lemma~\ref{lemma:q-atoms-in-Hp}
 we have that for $q$ sufficiently large, $H_{fin}^{\pp, q} (\R^n) =
 H_{fin}^{p_0, q} (w)$ as sets.   Given $f \in
 H_{fin}^{p_0, q} (w)$, define a weighted atomic
 decomposition norm on $H_{fin}^{p_0, q} (w)$ by 
	\[ \| f \|_{H_{fin}^{p_0, q} (w)} = \inf \bigg\{ \bigg\| 
\sum_{j = 1}^k \lambda_j^{p_0} \frac{\chi_{B_j}}{\|\chi_{B_j}\|_\pp ^{p_0}}
\bigg\|_{L^1 (w)}^{1/p_0} :
 f = \sum_{j = 1}^k \lambda_j a_j  \bigg\},  \] 
where the infimum is taken over all decompositions of f as
a finite sum of $(\pp,q)$  atoms.

\begin{lemma}\label{lemma:Weighted-FinDecomp} Given $\pp \in \MP$, fix
  $q$ as in Theorem~\ref{thm:qatoms} and let $w \in A_1 \cap L^{(\pp/p_0)'} (\R^n)$.
Then there exists $C = C(\pp, p_0, [w]_{A_1}, \| w \|_{(\pp/p_0)'})$ such that 
		\[ \| f \|_{H_{fin}^{p_0, q} (w)} \leq C \| f \|_{H^{p_0} (w)}. \]  
\end{lemma}

\begin{remark}
  We note in passing that Lemma~\ref{lemma:Weighted-FinDecomp} is not the
  same as \cite[Theorem 6.2]{MR2492226} because the
  the atoms given there are defined using the weighted $L^q$-norm, and we
  cannot pass between these two types of atoms simply by
  multiplying by a constant.
\end{remark}

\begin{proof}  
  The proof is very similar to the proof of Theorem~\ref{Var-Thm};
  here we sketch the changes required.  Fix $f \in H_{fin}^{p_0, q}
  (w)$; then $f \in H_{fin}^{\pp, q}(\R^n)$, and is supported on a
  ball $B=B(0, R)$ for some $R > 1$.  Let $\tilde{B} = B(0, 4R)$. By Lemma \ref{lemma:decay-Mg}, for $x
  \not\in \tilde{B}$, we have $\Mg f(x) \leq c \| \chi_{B}
  \|_{\pp}^{-1}$. 

 Assume that $\| f\|_{H^{p_0} (w)} = 1$;  we will show that  $\| f \|_{H_{fin}^{p_0, q}
  (w)} \leq C$.  By the proof of  \cite[Chapter 8, Theorem 1]{MR1011673} we have that 
\[ f = \sum_{k, j} \lambda_{k, j} a_{k, j} \]
where $\{a_{k,j}\}$ are  $(\pp,\infty)$ atoms supported on balls $B_{k,j}$, $\{\lambda_{k,j}\}$ are
non-negative, and 
\begin{equation} \label{eqn:wtd-atom}
 \bigg\| \sum_{k, j} \lambda_{k, j}^{p_0} \frac{\chi_{B_{k,j}}}{\|\chi_{B_{k, j}}\|_\pp^{p_0}}
\bigg\|_{L^1 (w)} \leq C \| f \|_{H^{p_0} (w)}^{p_0}.
\end{equation}
(As we noted in Remark~\ref{remark:atom-switch}, this is a restatement
of the results from \cite{MR1011673} to our setting.)
This decomposition is constructed in a fashion very similar to that of
Lemma~\ref{lemma:Atom-Stein} and the atoms and coefficients have much
the same properties.  Therefore, if we let $j'$ be the smallest
integer such that $2^{j'} \leq \|\chi_B\|_\pp^{-1}$ and 
regroup the sum as
	\[ f = \sum_{j \leq j'} \sum_k {\lambda}_{k, j} {a}_{k, j} 
+ \sum_{j > j'} \sum_k {\lambda}_{k, j} {a}_{k, j} = h + \ell,\]
the argument proceeds exactly as before.   This allows us to write 
\[ f = c (h/c) + (\ell - \ell_i) + \sum_{F_i} {\lambda}_{k, j} {a}_{k, j},  \]
where $h$ is a $(\pp,\infty)$ atom,  $i$ is chosen large enough that
$(\ell-\ell_i)$ is $(\pp,q)$ atom, and the sum is a finite sum of
$(\pp,\infty)$ atoms.  Moreover, we have that
\begin{multline*}
\| f \|_{H_{fin}^{p_0, q} (w)}^{p_0} 
\leq \bigg\| c^{p_0}\frac{\chi_{\tilde{B}}}{\|\chi_{\tilde{B}}\|_\pp^{p_0}}
+  \frac{\chi_{\tilde{B}}}{\|\chi_{\tilde{B}}\|_\pp^{p_0}}
 + \sum_{k,j} {\lambda}_{k, j}^{p_0} \frac{\chi_{B_{k,
       j}}}{\|\chi_{B_{k, j}}\|_\pp}  \bigg\|_{L^1 (w)} \\
\leq C\frac{w(\tilde{B})}{\|\chi_{\tilde{B}}\|_\pp^{p_0}} +
\bigg \| \sum_{k, j} {\lambda}_{k, j}^{p_0} \chi_{B_{k, j}} \bigg\|_{L^1 (w)}.
\end{multline*}
By \eqref{eqn:wtd-atom}, since the $\lambda_{k,j}$ are non-negative,
the last term is bounded by $C\|f\|_{H^{p_0}(w)}^{p_0} =C$.  To bound the first
term, note that 
by Lemmas~\ref{lemma:holder}
  and~\ref{lemma:homog-exp}, 
\begin{multline*}
 w (\tilde{B}) = \int_{\tilde{B}} w (x) dx 
\leq C(\pp,p_0)\| \chi_{\tilde{B}} \|_{\pp/p_0} \| w \|_{(\pp/p_0)'} \\
\leq C (\pp,p_0,\| w \|_{(\pp/p_0)'} )\| \chi_{\tilde{B}} \|_{\pp}^{p_0}.  
\end{multline*}
Hence, $\| f \|_{H_{fin}^{p_0, q} (w)} \leq C(\pp, p_0, [w]_{A_1}, \|
w \|_{(\pp/p_0)'})$ and our proof is complete.
\end{proof}

\section{Boundedness of operators on $H^{\pp}$} 
\label{section:sio}

In this section we show that convolution type Calder\'on-Zygmund
singular integrals with sufficient regularity are bounded on $H^\pp$.
Our two main techniques are the finite atomic decomposition from
Section~\ref{section:pq-atoms} and weighted norm inequalities.   First
we define the class of singular integrals we are interested in.

    \begin{definition} 	\label{eqn:kregular}
Let $K \in \mathcal{S}'$. We say $Tf = K \ast f$ is a convolution-type
singular integral operator with regularity of order $k$ if the distribution $K$ coincides with a
function on $\R^n\setminus \{0\}$ and has the following
properties:
\begin{enumerate}

\item  $\hat{K} \in L^{\infty}$;

\item for all multi-indices $0 \leq |\beta| \leq k + 1$ and $x \neq 0$,
$ \ds |\p^{\beta} K(x)| \leq \frac{C}{ |x|^{n+ |\beta|}}$.
    \end{enumerate}
\end{definition}

Singular integrals that satisfy this definition are bounded on $L^p$,
$1<p<\infty$.  More importantly, the pointwise smoothness conditions
guarantee that they satisfy weighted norm inequalities.  In
particular, we have the following weighted Kolmogorov inequality; for
a proof, see~\cite{duoandikoetxea01,garcia-cuerva-rubiodefrancia85}.

\begin{lemma} \label{lemma:kolmogorov}
Let $T$ be a convolution-type singular integral operator as defined
above. Given $w\in A_1$ and $0<p<1$, then for every ball $B$,
\[ \int_{B} |Tf(x)|^{p}w(x)\,dx
\leq C(T,n,p,[w]_{A_1})w(B)^{1-p}\left(\int_\subRn
  |f(x)|w(x)\,dx\right)^p. \] 
\end{lemma}

Our main results in this section are the following two theorems.

\begin{theorem} \label{thm:bdd-Lp} Given $\pp \in M\mathcal{P}_0$ and
  $q>1$ sufficiently large (as in Theorem \textup{\ref{thm:qatoms}}),
  let $T$ be a singular integral operator that has regularity of
  order $k \geq \lfloor n (\frac{1}{p_0} - 1) \rfloor $.  Then
\[ \|Tf\|_\pp \leq C(T,\pp,p_0,q,n)\|f\|_{H^\pp}. \]
\end{theorem}

\medskip

\begin{theorem} \label{thm:bdd-Hp}
Given $\pp \in M\mathcal{P}_0$ and
  $q>1$ sufficiently large (as in Theorem \textup{\ref{thm:qatoms}}),
  let $T$ be a singular integral operator that has regularity of
  order $k \geq \lfloor n (\frac{1}{p_0} - 1) \rfloor $.  Then
\[ \|Tf\|_{H^\pp} \leq C(T,\pp,p_0,q,n)\|f\|_{H^\pp}. \]
\end{theorem}

\medskip

We will prove both theorems as a consequence of a more general result
for sublinear operators. 

\begin{theorem} \label{thm:gen-bdd}
Given  $\pp \in M\mathcal{P}_0$ with $0<p_0<1$,  and  $q>1$ sufficiently large (as in
Theorem \textup{\ref{thm:qatoms}}),
suppose that $T$ is a sublinear operator that is defined on $(\pp,q)$
atoms.    Then:
\begin{enumerate}

\item  If  for all $w \in A_1\cap RH_{(q/p_0)'}$ and every $(\pp,
  q/p_0)$ atom $a(\cdot)$ with support $B$,
 \begin{equation}\label{Condition:Hp-Lp}
 \| Ta \|_{L^{p_0} (w)} \leq C(T,\pp,p_0,q,n,[w]_{A_1},[w]_{RH_{(q/p_0)'}}) \frac{w(B)^{1/p_0}}{\| \chi_{B} \|_{\pp}},
\end{equation}
then $T$ has a unique, bounded extension $\tilde{T} : H^{\pp}
\rightarrow L^{\pp}$.

\item If for all $w \in A_1\cap RH_{(q/p_0)'}$ and every $(\pp,
  q/p_0)$ atom $a(\cdot)$ with support  $B$,
\begin{equation}\label{Condition:Hp-Hp}
 \| Ta \|_{H^{p_0} (w)} \leq C(T,\pp,p_0,q,n,[w]_{A_1},[w]_{RH_{(q/p_0)'}}) \frac{w(B)^{1/p_0}}{\| \chi_{B} \|_{\pp}},
\end{equation}
then $T$ has a unique, bounded extension $\tilde{T} : H^{\pp}
\rightarrow H^{\pp}$.

\end{enumerate}
\end{theorem}

\begin{remark}
The additional hypothesis that $0<p_0<1$ is not a real restriction,
since by Lemma~\ref{lemma:max-up} we may take $p_0$ as small as
desired.
\end{remark}

\begin{remark} 
Note that when $\pp$ is constant and $w\equiv 1$, then
conditions~\eqref{Condition:Hp-Lp} and~\eqref{Condition:Hp-Hp} reduce
to showing that $T$ is uniformly bounded on atoms, which is the
condition used to prove singular integrals are bounded on
classical Hardy spaces.
\end{remark}

 \begin{proof}  
   First suppose that \eqref{Condition:Hp-Lp} holds.  Fix $f \in
   H_{fin}^{\pp, q/p_0}$; by Theorem~\ref{thm:qatoms} this set is
   dense in $H^\pp$.  Since $T$ is well-defined on the elements of
   $H_{fin}^{\pp, q/p_0}$, it will suffice to prove that
\begin{equation} \label{eqn:bdd-on-atoms}
 \|Tf\|_{\Lp} \leq C(T,\pp,p_0,q,n)\|f\|_{H^\pp}.
\end{equation}
For in this case by a standard density argument there exists a unique
bounded extension $\tilde{T}$ such that $\tilde{T} : H^{\pp}
\rightarrow L^{\pp}$.

To prove \eqref{eqn:bdd-on-atoms} we will use the 
extrapolation argument in Lemma \ref{lemma:ST-atoms} to reduce the
variable norm estimate to a weighted norm estimate.  Arguing
as we did in that proof, we have that 
	\[ \| Tf \|_{L^{\pp}}^{p_0} \leq \sup\int |Tf(x)|^{p_0} \mathcal{R}g(x) dx,  \] 
with the supremum taken over all $g \in L^{(\pp/p_0)'}$ with $\| g
\|_{(\pp/p_0)'} \leq 1$.   Suppose for the moment that we can prove
that for all $f\in H_{fin}^{\pp, q/p_0}$, 
\begin{equation}\label{Base:Hp-Lp}
\|Tf\|_{L^{p_0}(\mathcal{R} g)} \leq
C(T,\pp,p_0,q,n)\|f\|_{H^{p_0}(\mathcal{R} g)}. 
\end{equation}
(In particular, the constant is independent of $g$.)  Then we can
continue the argument as in the proof of Lemma \ref{lemma:ST-atoms} to get
 \begin{multline*}
 \| Tf \|_{L^{p_0} (\mathcal{R}g)}^{p_0} 
\leq C \| f \|_{H^{p_0} (\mathcal{R}g)}^{p_0} \leq C\int \Mg f(x)^{p_0} \mathcal{R}g(x) dx \\
\leq C \| (\Mg f)^{p_0} \|_{\pp/p_0} \ \| \mathcal{R}g \|_{(\pp/p_0)'} 
\leq C \| \Mg f \|_{\pp}^{p_0} \leq C\| f \|_{H^{\pp}}^{p_0}. 
\end{multline*}
This gives us~\eqref{eqn:bdd-on-atoms}.

\medskip

To complete the proof we will show \eqref{Base:Hp-Lp}.   Recall that
as sets, $H_{fin}^{p_0, q/p_0} (\mathcal{R}g) = H_{fin}^{\pp, q/p_0}$.
Therefore, let 
\[  f = \sum_{j = 1}^k {\lambda}_j {a}_j \]
be an arbitrary finite decomposition of $f$ in terms of $(\pp,q/p_0)$
atoms.   Since, $0 < p_0 < 1$,  by the sublinearity of $T$, convexity and \eqref{Condition:Hp-Lp},
\begin{multline*}
 \| Tf \|_{L^{p_0} (\mathcal{R}g)}^{p_0} = \int |Tf(x)|^{p_0} \mathcal{R}g (x) dx  
 \leq \sum_{j = 1}^k {\lambda}_j^{p_0} \int_{B_j} |T{a}_j (x)|^{p_0} \mathcal{R}g(x) dx \\
 \leq C \sum_{j = 1}^k {\lambda}_j^{p_0} \frac{\mathcal{R}g (B_j)}{\| \chi_{B_j} \|_{\pp}^{p_0}}
 = C\bigg\| \sum_{j = 1}^k {\lambda}_j^{p_0} \frac{\chi_{B_j}}{\|\chi_{B_j}\|_\pp^{p_0}} \bigg\|_{L^1 (\mathcal{R}g)}. 
 \end{multline*}
 This is true for any such decomposition of $f$.  Therefore, since
 $\mathcal{R}g \in A_1 \cap L^{(\pp/p_0)'}$ by construction,  by
 Lemma~\ref{lemma:Weighted-FinDecomp} we can take the infimum
 over all such decompositions to get $\| Tf
 \|_{L^{p_0} (\mathcal{R}g)} \leq C \| f \|_{H^{p_0} (\mathcal{R}g)}$, 
 where $C=C(T,\pp, p_0, q,n)$.   This proves 
 \eqref{Base:Hp-Lp} for all $f\in H_{fin}^{\pp, q/p_0}$.
	
\medskip

We now consider the case when condition~\eqref{Condition:Hp-Hp}
holds.  The proof is essentially the same as before, except instead of
proving \eqref{Base:Hp-Lp}, we need to prove that for all $f\in H_{fin}^{\pp, q/p_0}$, 
 \begin{equation}\label{Base:Hp-Hp}
\|Tf\|_{H^{p_0}(\mathcal{R} g)} \leq
C(T,\pp,p_0,q,n)\|f\|_{H^{p_0}(\mathcal{R} g)}. 
 \end{equation}
 Given this, we can then repeat the extrapolation argument as before.
 To prove~\eqref{Base:Hp-Hp} we use the same argument used to
 prove~\eqref{Base:Hp-Lp}, replacing $Tf$ with $\Mr (Tf)$ where
 $\Phi\in \Ss$ with $\int \Phi\,dx =1$, and using
 \eqref{Condition:Hp-Hp} instead of \eqref{Condition:Hp-Lp}.
	\end{proof}


\begin{proof}[Proof of Theorem~\ref{thm:bdd-Lp}]
By Theorem~\ref{thm:gen-bdd} it will suffice to show that
condition~\eqref{Condition:Hp-Lp} holds for all $(\pp,q/p_0)$ atoms and
all $w\in A_1\cap RH_{(q/p_0)'}$.  

Fix such an atom $a(\cdot)$ with support $B=B(x_0,r)$.  Let $2B =
B(x_0, 2r)$  and write
 \begin{multline*}
  \|Ta\|_{L^{p_0}(w)}^{p_0}
 = \int |Ta (x)|^{p_0} w (x)  dx \\
= \int_{2B} |Ta (x)|^{p_0} w(x) dx  + \int_{(2B)^c} |Ta (x)|^{p_0}
w(x) dx 
= I_1 + I_2.
\end{multline*}
We first estimate $I_1$. By Lemma~\ref{lemma:kolmogorov} there exists
a constant $C=C(T,n,p_0,[w]_{A_1})$ such that
 \begin{multline*}
 \int_{2B} |Ta (x)|^{p_0} w(x) dx 
\leq C w(B)^{1 - p_0} \left( \int_\subRn |a(x)| w(x) dx \right)^{p_0} \\
 \leq C w(B)^{1 - p_0} |B|^{p_0} \left( \avgint_B |a|^{q/p_0} dx \right)^{1/q} 
\left( \avgint_B w(x)^{(q/p_0)'} dx \right)^{p_0/(q/p_0)'}.
		\end{multline*}
Since $a(\cdot)$ is a $(\pp,q/p_0)$ atom and $w\in RH_{(q/p_0)'}$, we
get that 
\[ I_1 \leq C [w]_{RH_{(q/p_0)'}}^{p_0} \ w(B)^{1 - p_0} |B|^{p_0} \|
\chi_B \|_{L^{\pp}}^{-p_0} |B|^{-p_0} w(B)^{p_0} = C
[w]_{RH_{(q/p_0)'}}^{p_0} \ w(B) \| \chi_B \|_{L^{\pp}}^{-p_0}. \]

To estimate $I_2$, we start with a pointwise estimate.   Let $d = \lfloor n(\frac{1}{p_0} - 1) \rfloor$. 
We claim that there exists a constant $C = C(T, n)$ such that for all $x \in (2B)^c$,
\begin{equation}	\label{Ta-pointwise}
 |Ta(x)| \leq C \frac{|B|^{1 + \frac{d + 1}{n}}}{\| \chi_B \|_{\pp}} \cdot \frac{1}{|x - x_0|^{n + d + 1}}.
\end{equation}	
To prove this, let $P_d$ be the Taylor polynomial of $K$ of degree $d$
centered at $x - x_0$.  By our definition of $d$ and our assumption on $k$,
$d + 1 \leq k + 1$.   Therefore, the remainder $|K(x - y)| - P_d (y)|$ can be estimated by
Condition (2) in Definition~\ref{eqn:kregular}.  Hence, by the vanishing moment and size conditions on
$a(\cdot)$ and H\"older's inequality,
 \begin{align*}
 |Ta (x)|
 &\leq \int |K(x - y) - P_d (y)| |a(y)| dy  \\
& \leq \frac{C}{|x - x_0|^{n + d + 1}} \int_{B(x_0, r)} |y - x_0|^{d + 1} |a(y)| dy \\
 &\leq C \frac{r^{d + 1} |B| }{|x - x_0|^{n + d + 1}}\avgint_B a(y)\,dy \\
& \leq  C \frac{|B|^{\frac{n + d + 1}{n}} |B|^{-p_0/q}\|a\|_{q/p_0}}{|x
  - x_0|^{n + d + 1}} \\
& \leq  C \frac{|B|^{1 + \frac{d + 1}{n}}}{\| \chi_B \|_{\pp}} \cdot \frac{1}{|x - x_0|^{n + d + 1}}.
 \end{align*}
Given~\eqref{Ta-pointwise} we have that
\[	\int_{(2B)^c} |Ta(x)|^{p_0} w(x) dx 
 \leq C \frac{|B|^{p_0 (\frac{n + d + 1}{n})}}
{\| \chi_{B} \|_{\pp}^{p_0}} \underbrace{\int_{(2B)^c} \frac{w(x)}{|x - x_0|^{p_0 (n + d + 1)}} dx}_{J}. 
\]

To complete the proof we will show that there exists a constant $C = C(n, p_0)$ such that 
 \begin{equation}\label{Est:J}
 J \leq  C\frac{[w]_{A_1} w(B)}{|B|^{p_0 (\frac{n + d + 1}{n})}}.
\end{equation}
The proof of this is standard; for the convenience of the reader we
sketch the details.  Write  
\[ (2B)^c = \bigcup_{i =1}^{\infty} (2^{i + 1} B \backslash 2^i B); \]
then for $x \in 2^{i + 1} B \backslash 2^i B$, we have $|x - x_0|
\simeq 2^i r \simeq 2^i |B|^{1/n}$. Since $w \in A_1$ and $p_0 (n + d
+1) > n$, we can estimate as follows:
\begin{align*}
 J	
&= \sum_{i = 1}^{\infty} \int_{2^{i + 1} B \backslash 2^i B}  \frac{w(x)}{|x - x_0|^{p_0 (n + d + 1)}} dx \\ 
&\leq  \frac{C}{|B|^{p_0 (\frac{n + d + 1}{n})}} 
\sum_{i = 1}^{\infty} \frac{1}{2^{ip_0 (n + d + 1)}} \int_{2^{i + 1} B \backslash 2^i B} w(x) dx \\
 &=  \frac{C}{|B|^{p_0 (\frac{n + d + 1}{n})}} 
\sum_{i = 1}^{\infty} \frac{2^{n(i + 1)}|B| }{2^{ip_0 (n + d + 1)}} \ \avgint_{2^{i + 1} B} w(x) dx \\
 &\leq  \frac{C 2^n[w]_{A_1}}{|B|^{p_0 (\frac{n + d + 1}{n})}} 
\sum_{i = 1}^{\infty} \frac{1}{2^{ip_0 (n + d + 1)-in}} \left( |B| \ \essinf_{x \in B} w(x) \right) \\
& = \frac{C [w]_{A_1} w(B)}{|B|^{p_0 (\frac{n + d + 1}{n})}}.
	\end{align*}
This gives us~\eqref{Est:J} and so completes the proof.
	\end{proof}

\begin{proof}[Proof of Theorem~\ref{thm:bdd-Hp}]
  Our argument is similar to the proof of Theorem~\ref{thm:bdd-Lp}.
  By Theorem~\ref{thm:gen-bdd} it will suffice to show that
  condition~\eqref{Condition:Hp-Hp} holds for an arbitrary
  $(\pp,q/p_0)$ atom $a(\cdot)$ with support $B=B(x_0,r)$,  
  and all $w\in A_1\cap RH_{(q/p_0)'}$.  Fix $\Phi \in \mathcal{S}$
  with $\int \Phi = 1$; then we can estimate $\| Ta \|_{H^{p_0} (w)}$
  as follows:
	\begin{align*}
	\| Ta \|_{H^{p_0}(w)}^{p_0}
		&\lesssim \int_{2B} \Mr (Ta) (x)^{p_0} w(x) dx + \int_{(2B)^c} \Mr (Ta) (x)^{p_0} w(x) dx = J_1 + J_2. 
	\end{align*}
To estimate the $J_1$ we first use the fact that $\Mr (Ta) \leq c
M (Ta)$.  Moreover, we have that since $w\in A_1$, the
Hardy-Littlewood maximal operator also satisfies Kolmogorov's
inequality (see~\cite{duoandikoetxea01,garcia-cuerva-rubiodefrancia85}):
	\begin{align*}
	J_1 \leq C w(2B)^{1 - p_0}  \big(  \underbrace{\int_\subRn |Ta(x)| w(x) dx}_L \big)^{p_0}.
	\end{align*}
To get the desired estimate for $J_1$ it will suffice to show that
	\[ L = \int_\subRn |Ta(x)| w(x) dx \leq \frac{w (B)}{\| \chi_{B} \|_{\pp}}. \] 
To prove this, we again split the integral:
	\[ L = \int_{2B} |Ta(x)| w(x) dx + \int_{(2B)^c} |Ta(x)| w(x) dx = L_1 + L_2. \] 
To estimate $L_1$ we apply H\"{o}lder's inequality, the boundedness of
$T$ on $L^{q/p_0}$, and the fact that $w\in RH_{(q/p_o)'}$ to get 
 \begin{multline*}
 L_1 = \int_{2B} |Ta(x)| w(x) dx
 \leq \left( \int_{2B} |Ta(x)|^{q/p_0} dx \right)^{p_0/q} 
\left( \int_{2B} w(x)^{(q/p_0)'} dx \right)^{1/(q/p_0)'} \\
 \leq  \| a \|_{L^{q/p_0}} \cdot |2B|^{1/(q/p_0)'} \left( \avgint_{2B} w(x)^{(q/p_0)'}
   dx \right)^{1/(q/p_0)'} 
\leq C(n,[w]_{A_1},[w]_{RH_{(q/p_0)'}}) \frac{w(B)}{\| \chi_B \|_{\pp}}. 
	\end{multline*}
To estimate $L_2$ we repeat the argument we used to estimate $I_2$ in
the proof of Theorem~\ref{thm:bdd-Lp}, replacing the exponent $p_0$ by
$1$.  Then using the pointwise estimate for $Ta$ and the decomposition
argument, we have that 
	\begin{multline*}
	L_2 	\leq C \frac{|B|^{\frac{n + d + 1}{n}}}{\| \chi_B \|_{\pp}} \left( \int_{(2B)^c} \frac{w(x)}{|x - x_0|^{n + d + 1}} dx \right)  \\
		\leq C \frac{|B|^{\frac{n + d + 1}{n}}}{\| \chi_B \|_{\pp}} \cdot \frac{w(B) [w]_{A_1}}{|B|^{\frac{n + d + 1}{n}}} \cdot \left( \sum_{i = 0}^{\infty} \frac{2^{ni}}{2^{i(n + d + 1)}} \right) \leq C \frac{w(B)}{\| \chi_B \|_{\pp}}. 
	\end{multline*}

To estimate $J_2$, we will prove a pointwise bound for  $\Mr (Ta_j) (x)$ for
$x \in (2B_j)^c$ similar to~\eqref{Ta-pointwise}.  Define $K^{(t)} = K
\ast \Phi_t$;  then
$K^{(t)}$ satisfies condition (3) of Definition~\ref{eqn:kregular}
uniformly for all $t>0$.   Moreover, for $x\in (2B)^c$, the integral
for $K*a(x)$ converges absolutely, so
$|\Phi_t*(K*a)(x)|=|K^{(t)}*a(x)|$.  

Again let 
$d = \lfloor n(\frac{1}{p_0} - 1) \rfloor$ and fix $t>0$.  If $P_d$ is
the Taylor polynomial of $K^{(t)}$ centered at $x - x_0$, we can argue
exactly as we did to prove~\eqref{Ta-pointwise} to get
\begin{align*}
|K^{(t)} \ast a (x)|
 & = \left|\int [K^{(t)}(x - y) - P_d (y) ] a (y)\, dy\right| \\
 & \leq \frac{C}{ |x - x_0|^{n+ d + 1}} \int_{B(x_0,r)} |y-x_0|^{d + 1}|a(y)| dy \\
& \leq C\frac{|B|^{\frac{n+d+1}{n}}|B|^{-p_0/q}}{|x - x_0|^{n+ d + 1}}
\|a\|_{L^{q/p_0}} \\
& \leq  C\frac{|B|^{1+\frac{d + 1}{n}}}{\| \chi_{B} \|_\pp}\frac{1}{|x - x_0|^{n + d + 1}}.
\end{align*}
The final constant is independent of $t$, an so we can take the
supremum over all $t$ to
\[ \Mr (Ta)(x) \leq C\frac{ |B|^{1+\frac{d + 1}{n}}}{\| \chi_{B} \|_{\pp}}\frac{1}{|x - x_0|^{n + d + 1}}. \]
Then arguing as we did before, by \eqref{Est:J} we have that $J_2 \leq w(B)/ \| \chi_B
\|_{\pp}^{p_0}$. This completes the proof.
\end{proof}

\bibliographystyle{plain}
\bibliography{hardyvar}

\end{document}